\documentclass[twoside]{article}
\usepackage[cp1251]{inputenc}
\usepackage[T2A]{fontenc}
\usepackage{array}
\usepackage{amssymb}
\usepackage{amsmath}
\usepackage{amsthm}
\usepackage{latexsym}
\usepackage{bm}
\usepackage{enumerate}
\usepackage[dvips]{graphicx}
\usepackage{epsf}
\usepackage{wrapfig}
\usepackage{euscript}
\usepackage[english,russian]{babel}
\usepackage{textcomp}
\newtheorem{theorem}{Theorem}
\newtheorem{lemma}{Lemma}
\newtheorem{definition}{Definition}

\newtheorem{example}{Example}
\newtheorem{corollary}{Corollary}
\begin{document}
{\selectlanguage{english}
\binoppenalty = 10000 %
\relpenalty   = 10000 %

\pagestyle{headings} \makeatletter
\renewcommand{\@evenhead}{\raisebox{0pt}[\headheight][0pt]{\vbox{\hbox to\textwidth{\thepage\hfill \strut {\small Grigory. K. Olkhovikov}}\hrule}}}
\renewcommand{\@oddhead}{\raisebox{0pt}[\headheight][0pt]{\vbox{\hbox to\textwidth{{On generalized Van-Benthem-type characterizations}\hfill \strut\thepage}\hrule}}}
\makeatother

\title{On generalized Van-Benthem-type characterizations}
\author{Grigory K. Olkhovikov\\ Dept of Philosophy II, Ruhr-Universit\"{a}t Bochum\\Dept. of Philosophy, Ural Federal University
\\
tel.: +49-15123307070\\
email: grigory.olkhovikov@rub.de, grigory.olkhovikov@gmail.com}
\date{}
\maketitle

\begin{quote}
{\bf Abstract.} The paper continues the line of \cite{Ol13},
\cite{Ol14}, and \cite{Ol15}. This results in a model-theoretic
characterization of expressive powers of arbitrary finite sets of
guarded connectives of degree not exceeding $1$ and regular
connectives of degree $2$ over the language of bounded lattices.
\end{quote}

\begin{quote}
{\bf Keywords.} model theory, modal logic, intuitionistic logic,
propositional logic, asimulation, bisimulation, Van Benthem's
theorem.
\end{quote}

This paper is a further step in the line of our enquiries into the
expressive powers of intuitionistic logic and its extensions. This
line started in late 2011, when we began to think about possible
modifications of bisimulation relation in order to obtain the full
analogue of Van Benthem modal characterization theorem for
intuitionistic propositional logic. For the resulting
modification, which was published in \cite{Ol13}, we came up with
a term ``asimulation'', since one of the differences between
asimulations and bisimulations was that asimulations were not
symmetrical.

Later we modified and extended asimulations in order to capture
the expressive powers of first-order intuitionistic logic (in
\cite{Ol14}) and some variants of basic modal intuitionistic logic
(in \cite{Ol15}) viewed as fragments of classical first-order
logic. Some other authors were also working in this direction;
e.g. in \cite{Badia2015} this line of research is extended to
bi-intuitionistic propositional logic, although the author prefers
directed bisimulations to asimulations.

In this paper we publish a general algorithm allowing for an easy
computation of asimulation-like notions for a class of fragments
of classical first-order logic that can be naturally viewed as
induced by some kind of intensional propositonal logic via the
corresponding notion of standard translation. The group of
appropriate intensional logics includes all of the above mentioned
logics (except, for obvious reasons, the first-order
intuitionistic logic) but also many other formalisms. It is worth
noting that not all of these formalisms are actually extensions of
intuitionistic logic, in fact, even the classical modal
propositional logic which is the object of the original Van
Benthem modal characterization theorem\footnote{For its
formulation see, e.g. \cite[Ch.1, Th. 13]{Blackburn2006}.}, is
also in this group. Thus the generalized asimulations defined in
this paper have an equally good claim to be named generalized
bisimulations, and if we still continue to call them asimulations,
we do it mainly because for us these relations and their use are
inseparable from the above-mentioned earlier results on the
expressive power of intuitionistic logic.

The rest of this paper has the following layout. Section
\ref{S:Prel} fixes the main preliminaries in the way of notation
and definition. In Section \ref{S:bool} we give some simple facts
about Boolean functions and define the notion of a standard
fragment of correspondence language. In Section \ref{S:Asim} we do
the main technical work preparing the `easy' direction of our
generalization of Van Benthem modal characterization theorem and
define our central notion of (generalized) asimulation.  In
Section \ref{S:Sat} we do the technical work for the `hard'
direction which mainly revolves around the properties of
asimulations over $\omega$-saturated models. Section \ref{S:Main}
contains the proof of the result itself, and Section \ref{S:Con}
gives conclusions, discusses the limitations of the result
presented and prospects for future research.

\section{Preliminaries}\label{S:Prel}

We consider the correspondence language, which is a first-order
language without identity over the vocabulary $\Sigma =
\{\,R^2_1,\ldots R^2_n,\ldots, P_1^1,\ldots P_n^1,\ldots\,\}$. A
\emph{formula} is a formula of the correspondence language. A
model of the correspondence language is a classical first-order
model of signature $\Sigma$. We refer to correspondence formulas
with lower-case Greek letters $\theta$, $\tau$, $\varphi$, $\psi$,
and $\chi$, and to sets of correspondence formulas with upper-case
Greek letters $\Gamma$ and $\Delta$.

We will use items from the following list to denote individual
variables:

$$
x_1, y_1, z_1, w_1,\ldots, x_n, y_n, z_n, w_n,\ldots
$$

We will write $x, y, z, w$ as a shorthand for $x_1, y_1, z_1,
w_1$.

We denote the set of natural numbers by $\omega$.

If $\varphi$ is a correspondence formula, then we associate with
it the following vocabulary $\Sigma_\varphi \subseteq \Sigma$ such
that $\Sigma_\varphi = \{\,R_1,\ldots, R_n, \ldots\,\} \cup
\{\,P_i \mid P_i \text{ occurs in }\varphi\,\}$. More generally,
we refer with $\Theta$ to an arbitrary subset of $\Sigma$ such
that $\{\,R_1,\ldots, R_n, \ldots\,\} \subseteq \Theta$. If $\psi$
is a formula and every predicate letter occurring in $\psi$ is in
$\Theta$, then we call $\psi$ a $\Theta$-formula.

We refer to a sequence $o_1,\dots, o_n$ of any objects as
$\bar{o}_n$. We identify a sequence consisting of a single element
with this element. If all free variables of a formula $\varphi$
(formulas in $\Gamma$) occur in $\bar{x}_n$, we write
$\varphi(\bar{x}_n)$ ($\Gamma(\bar{x}_n)$).

We use the following notation for models of classical predicate
logic:
\[
M = \langle U, \iota\rangle, M_1 = \langle U_1, \iota_1\rangle,
M_2 = \langle U_2, \iota_2\rangle,\ldots , M' = \langle U',
\iota'\rangle, M'' = \langle U'', \iota''\rangle,\ldots,
\]
where the first element of a model is its domain and the second
element is its interpretation of predicate letters. If $k \in
\omega$ then we write $R^k_n$ as an abbreviation for
$\iota_k(R_n)$. If $a \in U$ then we say that $(M, a)$ is a
pointed model. Further, we say that $\varphi(x)$ is true at $(M,
a)$ and write $M, a \models \varphi(x)$ iff for any variable
assignment $\alpha$ in $M$ such that $\alpha(x) = a$ we have $M,
\alpha \models \varphi(x)$. It follows from this convention that
the truth of a formula $\varphi(x)$ at a pointed model is to some
extent independent from the choice of its only free variable.
Moreover, for $k \in \omega$ we will sometimes write $a \models_k
\varphi(x)$ instead of $M_k, a \models \varphi(x)$.

In what follows we will need a notion of \emph{$k$-ary guarded
$x$-connective} ($k$-ary $x$-g.c.) for a given variable $x$ in the
correspondence language. Such a connective is a formula
$\varphi(x)$ of a special form, which we define inductively as
follows:

1. $\mu = \psi(P_1(x),\ldots, P_k(x))$ is a $k$-ary guarded
$x$-connective of degree $0$ iff $\psi(P_1(x),\ldots, P_k(x))$ is
a Boolean combination of $P_1(x),\ldots, P_k(x)$, that is to say,
a formula built from $P_1(x),\ldots, P_k(x)$ using only $\wedge,
\vee$, and $\neg$. In this case $\mu$ is neither $\forall$-guarded
nor $\exists$-guarded.

2. If $\mu^-$ is a $k$-ary guarded $x_{m + 1}$-connective of
degree $n$, and $\mu^-$ is not $\forall$-guarded, then, for
arbitrary $S_1,\ldots, S_m \in \{\,R_1,\ldots, R_n, \ldots\,\}$,
formula

\begin{equation}\label{E:Aguard}\tag{\text{$\forall$-guard}}
\forall x_2\ldots x_{m + 1}((\bigwedge^{m}_{i = 1}(S_i(x_i,x_{i +
1})) \to \mu^-)
\end{equation}

is a $k$-ary $\forall$-guarded $x$-connective of degree $n + 1$,
provided that it is not equivalent to a $k$-ary guarded
$x$-connective of a smaller degree.

3. If $\mu^-$ is a $k$-ary guarded $x_{m + 1}$-connective of
degree $n$, and $\mu^-$ is not $\exists$-guarded, then, for
arbitrary $S_1,\ldots, S_m \in \{\,R_1,\ldots, R_n, \ldots\,\}$,
formula

\begin{equation}\label{E:Eguard}\tag{\text{$\exists$-guard}}
\exists x_2\ldots x_{m + 1}(\bigwedge^{m}_{i = 1}(S_i(x_i,x_{i +
1})) \wedge \mu^-)
\end{equation}

is a $k$-ary $\exists$-guarded $x$-connective of degree $n + 1$,
provided that it is not equivalent to a $k$-ary guarded
$x$-connective of a smaller degree.

Thus degree of a $k$-ary $x$-g.c. is just the number of quantifier
alternations in it. Degree of a $k$-ary $x$-g.c. $\mu$ we will
abbreviate as $\delta(\mu)$. The modality $\mu^-$ mentioned in the
above definition is called the immediate ancestor of $\mu$. If
$\delta(\mu) > 0$, then $\mu^-$ always exists, and, moreover, we
have $\delta(\mu^-) = \delta(\mu) - 1$. Taking the transitive
closure of immediate ancestry relation, we obtain that for every
$k$-ary $x$-g.c. $\mu$ there exist $\delta(\mu)$ ancestors, which
we will denote $\mu^0,\ldots, \mu^{\delta(\mu) - 1}$ respectively,
assuming that in this sequence the $x$-g.c.'s with smaller
superscripts are ancestors also of the $x$-g.c.'s with bigger
superscripts, so that $\mu^- = \mu^{\delta(\mu) - 1}$, $(\mu^-)^-
= \mu^{\delta(\mu) - 2}$, etc. In this sequence every $\mu^i$ is a
guarded connective of degree $i$, so that $\mu^0$ always defines a
Boolean function for the corresponding set of atoms. For a given
$\mu$, we will call $\mu^0$ the \emph{propositional core} of
$\mu$.

Since in this paper we are interested in the expressive powers of
guarded connectives, we will lump together different guarded
connectives which are equivalent as formulas in the correspondence
language, treating them as one and the same connective. In the
same spirit, we are not going to distinguish between Boolean
functions which are equivalent as formulas, and only differ from
one another due to presence of vacuous variables. Thus every
$m$-ary Boolean function under this angle of view \emph{is} at the
same time an $n$-ary Boolean function for $m \leq n$ for which $n
- m$ extra variables are treated as vacuous.

\begin{example}
{\em
We list some examples of $x$-g.c.'s:

$1$. Standard connectives $\bot$, $\top$, $\neg$, $\wedge$,
$\vee$, $\to$, $\leftrightarrow$ in their classical reading are
all, when applied to $P_1(x)$ and $P_2(x)$, examples of $x$-g.c.'s
of degree $0$ and of corresponding arity.

$2$. Examples of unary $\forall$-guarded $x$-g.c.'s are:

$$
\lambda_1 = \forall y(R_1(x,y) \to P_1(y));
$$
$$
\lambda_2 = \forall yz(R_1(x,y) \wedge R_2(y,z)) \to P_1(z));
$$
$$
\lambda_3 = \forall y(R_1(x,y) \to\exists z(R_3(y,z) \wedge
P_1(z))).
$$
The last of these g.c's has degree $2$, the others have degree
$1$.

$3$. The following formula is an example of unary
$\exists$-guarded $x$-g.c. of degree $1$:
$$
\lambda_4 = \exists y(R_1(x,y) \wedge P_1(y)).
$$

$4$. Finally, an example of binary $\forall$-guarded $x$-g.c. of
degree $1$:
$$
\lambda_5 = \forall y(R_1(x,y) \to (\neg P_1(y) \vee P_2(y))).
$$
}
\end{example}

In what follows we will frequently encounter the long conjunctions
similar to those in the definition of guarded connective above.
Therefore, we introduce for them special notation. If $k < l$,
$z_k,\ldots, z_{l + 1}$ variables, and $S_k,\ldots, S_m \in
\{\,R_1,\ldots, R_n, \ldots\,\}$, then we abbreviate
$\bigwedge^l_{i = k}(S_i(z_i,z_{i + 1}))$ by $\pi^l_kSz$.
Similarly, if $c_k,\ldots, c_{l + 1}$ is a sequence of elements of
$U_r$, then we abbreviate $\bigwedge^l_{i =
k}(\iota_r(S_i)(c_i,c_{i + 1}))$ by $\pi^l_kS^rc$. In this
notation, the above formulas \eqref{E:Aguard} and \eqref{E:Eguard}
will look as

$$
\forall x_2\ldots x_{m + 1}(\pi^{m}_1Sx \to \mu^-),
$$

and

$$
\exists x_2\ldots x_{m + 1}(\pi^{m}_1Sx \wedge \mu^-),
$$
respectively.

It is obvious, that for natural $r,s$ such that $r < s$, every
$r$-ary $x$-g.c. is an $s$-ary $x$-g.c.

A guarded $x$-connective ($x$-g.c.) is a $k$-ary guarded
$x$-connective for some $k \geq 0$.

If $\varphi_1(z),\ldots, \varphi_k(z)$ are formulas in the
correspondence language, each with a single free variable, and
$\mu$ is a $k$-ary $x$-g.c., then we call the \emph{application of
$\mu$ to $\varphi_1,\ldots, \varphi_k$} the result of replacing
every formula in $P_1(w),\ldots, P_k(w)$ for some variable $w$ in
$\mu$ by formulas $\varphi_1(w),\ldots, \varphi_k(w)$,
respectively, and we denote the resulting formula by
$\mu(\varphi_1,\ldots, \varphi_k)$.

If $x$ is a variable in the correspondence language, then we say
that the set $\mathcal{L}^\Theta_x(\mathbb{M})$ of formulas in
variable $x$ is a \emph{guarded $x$-fragment} of the
correspondence language iff $\mathbb{M}$ is a set of $x$-g.c's and
$\mathcal{L}^\Theta_x(\mathbb{M})$ is the least set of
$\Theta$-formulas, such that $P_n(x) \in
\mathcal{L}^\Theta_x(\mathbb{M})$ for every $P_n \in \Theta$, and
$\mathcal{L}^\Theta_x(\mathbb{M})$ is closed under applications of
$x$-g.c's from $\mathbb{M}$. If $\mathbb{M} = \{ \mu_1,\ldots,
\mu_s \}$ is a finite set of $x$-g.c.'s, then we write
$\mathcal{L}^\Theta_x(\mathbb{M})$ as $\mathcal{L}_x(\mu_1,\ldots,
\mu_s)$.

\begin{example}
{\em

We list some examples of guarded $x$-fragments using the notation
of the previous example:
\begin{enumerate}
\item $\mathcal{L}^\Sigma_x(\wedge, \vee, \bot, \top, \lambda_5)$
is the set of all standard $x$-translations of propositional
intuitionistic formulas.

\item $\mathcal{L}^\Sigma_x(\wedge, \vee, \bot, \top, \neg,
\lambda_1, \lambda_4)$ is the set of all standard $x$-translations
of propositional modal formulas.

\item $\mathcal{L}^\Sigma_x(\wedge, \vee, \bot, \top, \lambda_2,
\lambda_3, \lambda_5)$ is the set of all standard $x$-translations
of propositional modal intuitionistic formulas, in case we assume
for intuitionistic modal logic the type of Kripke semantics
defined by clauses ($\Box_2$) and ($\Diamond_2$) of \cite[Section
4]{AlShk06}.
\end{enumerate}
}
\end{example}

\section{Standard fragments and classification of Boolean
functions}\label{S:bool}

In this paper we are interested in characterizing the expressive
powers of guarded $x$-fragments $\mathcal{L}^\Theta_x(\mathbb{M})$
of the correspondence language, such that $\{ \wedge, \vee, \top,
\bot \} \subseteq \mathbb{M}$, by means of an invariance with
respect to a suitable class of binary relations. Therefore, we
need to define the respective type of invariance property:
\begin{definition}\label{D:invariance}
{\em Let $\alpha$ be a class of relations such that for any $A \in
\alpha$ there is a $\Theta$ and there are $\Theta$-models $M_1$
and $M_2$ such that the following condition holds:
\begin{align}
&A \subseteq (U_1 \times U_2) \cup (U_2\times
U_1)\label{E:c-typ}\tag{\text{type}}
\end{align}
Then a formula $\varphi(x)$ is said to be \emph{invariant with
respect to $\alpha$}, iff for every $A \in \alpha$ for the
corresponding $\Theta$-models $M_1$ and $M_2$, and for arbitrary
$a \in U_1$ and $b \in U_2$ it is true that:
\[
(a\mathrel{A}b \wedge a \models_1 \varphi(x)) \Rightarrow b
\models_2 \varphi(x).
\]
}
\end{definition}
The above definition defines formula invariance under rather
special conditions. However, these conditions will hold for all
the binary relations to be considered below, therefore this
definition suits our purposes.

If a formula is invariant w.r.t. a singleton $\{ A \}$, we simply
say that a formula is invariant w.r.t. $A$. Clearly, a formula is
invariant w.r.t. a class $\alpha$ iff this formula is invariant
w.r.t. every $A \in \alpha$. We say that a set $\Gamma(x)$ is
invariant w.r.t. $\alpha$ iff every formula in $\Gamma(x)$ is
invariant w.r.t. $\alpha$.

Our purpose in the present paper, therefore, is to give an
algorithm which, for a given guarded $x$-fragment
$\mathcal{L}^\Theta_x(\mathbb{M})$ of the correspondence language
would compute a definition of a class of binary relations such
that a formula $\psi(x)$ of the correspondence language is
equivalent to a formula in $\mathcal{L}^\Theta_x(\mathbb{M})$ iff
it is invariant w.r.t. this class. Members of the respective
classes of binary relations we will call \emph{asimulations}.

However, one can expect that not all guarded fragments are equally
amenable to such a treatment. Therefore, the rest of this section
is devoted to isolating some special well-behaved subsets of the
class of guarded connectives. The guarded fragments generated by
such well-behaved subsets we are going to designate as `standard'
ones and claim them as the proper scope of the general algorithm
mentioned in the previous paragraph.

The definition of a guarded connective suggests $2$ natural
rubrics for their classification, the one according to their
degree and the other according to the type of Boolean function
defined by their propositional core. We are now going to look
closer into the latter.

First let us mention the natural order on the doubleton set of
classical truth values: $0 < 1$. This order induces the natural
order on the $n$-tuples of truth values for which we have
$\bar{a}_n \leq \bar{b}_n$ iff for every $i$ between $1$ and $n$
we have $a_i \leq b_i$ as truth values. We now define the
following types of Boolean functions:

\begin{enumerate}
\item \emph{Monotone} functions. A Boolean $n$-ary function $f$ is
monotone iff for all $n$-tuples of truth values $\bar{a}_n$ and
$\bar{b}_n$ we have
$$
\bar{a}_n \leq \bar{b}_n \Rightarrow f(\bar{a}_n) \leq
f(\bar{b}_n).
$$

\item \emph{Anti-monotone} functions. A Boolean $n$-ary function
$f$ is anti-monotone iff for all $n$-tuples of truth values
$\bar{a}_n$ and $\bar{b}_n$ we have
$$
\bar{a}_n \leq \bar{b}_n \Rightarrow f(\bar{a}_n) \geq
f(\bar{b}_n).
$$

\item \emph{Rest} functions. A Boolean function is a rest function
iff it is neither monotone nor anti-monotone.

\item $TFT$-functions. A Boolean $n$-ary function $f$ is a
$TFT$-function iff there exist three $n$-tuples of truth values
$\bar{a}_n$, $\bar{b}_n$, and $\bar{c}_n$ such that (1) $\bar{a}_n
< \bar{b}_n < \bar{c}_n$, and (2) $f(\bar{a}_n) = f(\bar{c}_n) =
1$, whereas $f(\bar{b}_n) = 0$.

\item $FTF$-functions. A Boolean $n$-ary function $f$ is an
$FTF$-function iff there exist three $n$-tuples of truth values
$\bar{a}_n$, $\bar{b}_n$, and $\bar{c}_n$ such that (1) $\bar{a}_n
< \bar{b}_n < \bar{c}_n$, and (2) $f(\bar{a}_n) = f(\bar{c}_n) =
0$, whereas $f(\bar{b}_n) = 1$.

\end{enumerate}

Note that under this reading the class of monotone functions has a
non-empty intersection with the class of anti-monotone functions,
which consists of constant functions. Further, note that all
$TFT$-functions and $FTF$ functions are \emph{ex definitione} rest
functions. The rest functions which are \emph{not} $TFT$-functions
we will call $\forall$\emph{-special}. Similarly, the rest
functions which are \emph{not} $FTF$-functions we will call
$\exists$\emph{-special}.\footnote{A rest function can be neither
$\forall$- nor $\exists$-special. Take, for instance, $p_1
\leftrightarrow p_2 \leftrightarrow p_3$ and consider the
following series of tuples:
$$
(0,0,0) < (1,0,0) < (1,1,0) < (1,1,1).
$$
However, it is impossible for a rest function to be both
$\forall$-special and $\exists$-special. Indeed, if $f$ is a rest
function then take $\bar{a}_n$, $\bar{b}_n$, $\bar{c}_n$,
$\bar{d}_n$ such that $\bar{a}_n < \bar{b}_n$ and $\bar{c}_n <
\bar{d}_n$ for which we have
$$
f(\bar{a}_n) = f(\bar{d}_n) = 0; f(\bar{b}_n) = f(\bar{c}_n) = 1.
$$
Since $f(\bar{a}_n) \neq f(\bar{c}_n)$, we must have $\bar{a}_n
\neq \bar{c}_n$. Therefore, if $\bar{a}_n$ and $\bar{c}_n$ are
comparable, then we must have either $\bar{a}_n < \bar{c}_n$ or
$\bar{c}_n < \bar{a}_n$. In the former case the sequence
$(\bar{a}_n, \bar{c}_n, \bar{d}_n)$ shows that $f$ is an
$FTF$-function, whereas in the latter case the sequence
$(\bar{c}_n, \bar{a}_n, \bar{b}_n)$ shows that $f$ is a
$TFT$-function. Finally, if $\bar{a}_n$ and $\bar{c}_n$ are
incomparable, one has to consider $\bar{e}_n = min(\bar{a}_n,
\bar{c}_n)$. We must have then $\bar{e}_n < \bar{a}_n, \bar{c}_n$,
and, depending on the value of $f(\bar{e}_n)$, we get that either
the sequence $(\bar{e}_n, \bar{c}_n, \bar{d}_n)$ verifies that $f$
is an $FTF$-function or the sequence $(\bar{e}_n, \bar{a}_n,
\bar{b}_n)$ verifies that $f$ is a $TFT$-function.} Further,
Boolean functions (not necessarily rest functions) which are
\emph{not} $FTF$-functions we will call \emph{weakly}
$\exists$\emph{-special}, whereas Boolean functions (again, not
necessarily rest functions) which are \emph{not} $TFT$-functions
we will call \emph{weakly} $\forall$\emph{-special}. Thus, every
non-rest function is both weakly $\forall$-special and weakly
$\exists$-special.

Before we go any further, we need to introduce a more convenient
notation for subclasses of guarded connectives which emphasizes
both the structure of their quantifier prefix and the type of
their propositional core. Classes of $x$-g.c.'s will be denoted by
expressions of the form $\nu_x(Q_1 \ldots Q_k, f)$ where $f$ is a
Boolean function and $Q_1 \ldots Q_k$ is a possibly empty sequence
of alternating quantifiers from the set $\{ \forall, \exists \}$.
Thus, $\nu_x(\emptyset, f)$ denotes a class of all $x$-g.c.'s of
degree $0$ which define a Boolean function $f$ for their atomic
components. If the meaning of $\nu_x(Q_1 \ldots Q_k, f)$ is
already defined, and $Q \in \{ \forall, \exists \}$ is such that
$Q \neq Q_1$, we further define that  $\nu_x(QQ_1 \ldots Q_k, f)$
is the class of all $Q$-guarded $x$-g.c.'s $\mu$ for which we have
$\mu^- \in \nu_x(Q_1 \ldots Q_k, f)$.

One has to note that at least in case of constants this notation
can be misleading since we have for example
$$
\nu_x(\forall, \top) = \{ \top \}; \nu_x(\exists, \bot) = \{ \bot
\},
$$
and thus these two classes are not actually classes of g.c.'s of
degree $1$. And in general the classes $\nu_x(Q_1 \ldots
Q_k\exists, \bot)$ and $\nu_x(Q_1 \ldots Q_k\forall, \top)$ always
coincide with the classes $\nu_x(Q_1 \ldots Q_k, \bot)$ and
$\nu_x(Q_1 \ldots Q_k, \top)$, respectively. Therefore, we omit
classes of the forms $\nu_x(Q_1 \ldots Q_k\exists, \bot)$ and
$\nu_x(Q_1 \ldots Q_k\forall, \top)$ from our classification.

This phenomenon, however, does not seem to arise with the other
pieces of the introduced notation: for guarded connectives with
non-constant $f$ the length of quantifier prefix in $\nu_x(Q_1
\ldots Q_k, f)$ is precisely the same as the degree of the
elements of $\nu_x(Q_1 \ldots Q_k, f)$ and different such $\nu$'s
denote different and disjoint classes of guarded
connectives.\footnote{The only possible difficulty seems to
concern the vacuous arguments of Boolean functions. We avoid this
by lumping together Boolean functions defined by logical
equivalents.} This is certainly so for the classes of guarded
connectives of degree not exceeding $2$ which will mostly concern
us below.

With this $\nu$-notation, we can provide a concise description for
further important subclasses of guarded connectives. We define
that an $x$-g.c. $\mu$ is \emph{special} iff for some variable
$x$, $\mu$ is in the class $\nu_x(Q, f)$ and $f$ is $Q$-special.
Moreover, we define, that an $x$-g.c. $\mu$ is \emph{weakly
special} iff $\mu \in \nu_x(Q, f)$ and $f$ is weakly $Q$-special.

We now want to designate the following special classes of guarded
connectives:

\begin{enumerate}
\item \emph{Flat connectives} are guarded connectives of degree
less or equal to $1$.

\item \emph{Modalities} are guarded connectives with a
propositional core, defining a non-constant Boolean function which
is either monotone or anti-monotone.

\item \emph{Regular connectives} are guarded connectives with a
weakly special ancestor of degree $1$ and a non-constant
propositional core.

\end{enumerate}

It is easy to see that these three classes have pairwise non-empty
intersection. For example, the intersection of flat connectives
and regular connectives is the class of all weakly special
connectives, and given that propositional cores of modalities are
both weakly $\forall$-special and weakly $\exists$-special, we get
that every modality is regular.

Before we turn to the definition of a standard fragment, we
collect some facts about Boolean functions:

\begin{lemma}\label{L:Boolean}
Let $f(p_1,\ldots, p_n)$ be a Boolean function. Then:
\begin{enumerate}
\item If $f$ is non-constant monotone, then it is expressible as
$F(p_1,\ldots, p_n)$, where $F$ is a superposition of $\wedge$'s
and $\vee$'s.

\item If $f$ is non-constant monotone, then $f(p_1,\ldots, p_1)$
is just $p_1$.

\item If $f$ is non-constant anti-monotone, then it is expressible
as $\neg F(p_1,\ldots, p_n)$, where $F$ is a superposition of
$\wedge$'s and $\vee$'s.

\item If $f$ is non-constant anti-monotone, then $f(p_1,\ldots,
p_1)$ is just $\neg p_1$.

\item If $f$ is a $TFT$-function, then, for some $A_1,\ldots, A_n
\in \{ p_1 , p_2, p_1 \vee p_2, \bot, \top \}$, $f(A_1,\ldots,
A_n)$ is equivalent to $p_1 \to p_2$.

\item If $f$ is an $FTF$-function, then, for some $A_1,\ldots, A_n
\in \{ p_1 , p_2, p_1 \wedge p_2, \bot, \top \}$, $f(A_1,\ldots,
A_n)$ is equivalent to $p_1 \wedge \neg p_2$.

\item If $f$ is a rest function, then for some $B_1,\ldots, B_n,
C_1,\ldots, C_n \in \{ p_1, \bot, \top \}$ we have both
$f(B_1,\ldots, B_n)$ equivalent to $p_1$ and $f(C_1,\ldots, C_n)$
equivalent to $\neg p_1$.

\item If $f$ is a non-constant non-$FTF$ function, then $f$ admits
of representation in the following form:
$$
(p^1_1\wedge,\ldots, \wedge p^1_{i_1}) \vee, \ldots, \vee
(p^k_1\wedge,\ldots, \wedge p^k_{i_k}) \vee (\neg
q^1_1\wedge,\ldots, \wedge \neg q^1_{j_1}) \vee, \ldots, \vee
(\neg q^m_1\wedge,\ldots, \wedge \neg q^m_{j_m})
$$
where $m, k \geq 0$,  $m + k > 0$, $i_1,\ldots, i_k, j_1,\ldots,
j_m > 0$ and
$$
p^1_1,\ldots, p^1_{i_1}, \ldots, p^k_1,\ldots, p^k_{i_k},
q^1_1,\ldots, q^1_{j_1} \in \{ p_1,\ldots, p_n \}.
$$

\item If $f$ is a non-constant non-$TFT$ function, then $f$ admits
of representation in the following form:
$$
(p^1_1\vee,\ldots, \vee p^1_{i_1}) \wedge, \ldots, \wedge
(p^k_1\vee,\ldots, \vee p^k_{i_k}) \wedge (\neg q^1_1\vee,\ldots,
\vee \neg q^1_{j_1}) \wedge, \ldots, \wedge (\neg
q^m_1\vee,\ldots, \vee \neg q^m_{j_m})
$$
where $m, k \geq 0$,  $m + k > 0$, $i_1,\ldots, i_k, j_1,\ldots,
j_m > 0$ and
$$
p^1_1,\ldots, p^1_{i_1}, \ldots, p^k_1,\ldots, p^k_{i_k},
q^1_1,\ldots, q^1_{j_1} \in \{ p_1,\ldots, p_n \}.
$$
\end{enumerate}
\end{lemma}
\begin{proof} Parts $1$ through $4$ are all just known basic facts about
Boolean functions. We concentrate on the parts $5$ through $9$.

As for Part $5$, assume that $f$ is a $TFT$-function. Then
(renumbering $p$'s if necessary) for some $1 \leq k < l < m \leq
n$ we must have all of the following:
\begin{align}
(p_1 \wedge \ldots \wedge p_{k -1} \wedge \neg p_k \wedge \ldots
\wedge
\neg p_n) &\to f(p_1,\ldots, p_n);\label{E:bool1}\\
(p_1 \wedge \ldots \wedge p_l \wedge \neg p_{l +1} \wedge \ldots
\wedge
\neg p_n) &\to \neg f(p_1,\ldots, p_n);\label{E:bool2}\\
(p_1 \wedge \ldots \wedge p_m \wedge \neg p_{m +1} \wedge \ldots
\wedge \neg p_n) &\to f(p_1,\ldots, p_n).\label{E:bool3}
\end{align}

We have then one of the two cases: either
$$
(p_1 \wedge \ldots \wedge p_{k -1} \wedge \neg p_k \wedge \ldots
\wedge \neg p_l \wedge  p_{l +1} \wedge \ldots \wedge p_m \wedge
\neg p_{m +1} \wedge \ldots \wedge \neg p_n) \to f(p_1,\ldots,
p_1),
$$
or
$$
(p_1 \wedge \ldots \wedge p_{k -1} \wedge \neg p_k \wedge \ldots
\wedge \neg p_l \wedge  p_{l +1} \wedge \ldots \wedge p_m \wedge
\neg p_{m +1} \wedge \ldots \wedge \neg p_n) \to \neg
f(p_1,\ldots, p_1).
$$

In the first case we set as follows:
$$
A_1, \ldots, A_{k -1} := \top;
$$
$$
A_k, \ldots, A_l := p_1;
$$
$$
A_{l +1}, \ldots, A_m := p_2;
$$
$$
A_{m +1}, \ldots, A_n := \bot.
$$

In the second case the settings are as follows:
$$
A_1, \ldots, A_{k -1} := \top;
$$
$$
A_k, \ldots, A_l := p_1 \vee p_2;
$$
$$
A_{l +1}, \ldots, A_m := p_1;
$$
$$
A_{m +1}, \ldots, A_n := \bot.
$$

One can straightforwardly verify then, that in each of the two
cases the respective settings for $A_1,\ldots, A_n$ give us
$$
f(A_1,\ldots, A_n) \Leftrightarrow p_1 \to p_2.
$$

Part $6$ is just a dual of Part $5$.

As for Part $7$, assume that $f$ is a rest function, that is, $f$
is neither monotone nor anti-monotone. This means that (renaming
$p$'s if necessary) there exist $1 \leq k < l \leq n$ and also $1
\leq k' < l' \leq n$ for which all of the following holds:
\begin{align}
(p_1 \wedge \ldots \wedge p_{k -1} \wedge \neg p_k \wedge \ldots
\wedge
\neg p_n) &\to f(p_1,\ldots, p_n);\label{E:bool4}\\
(p_1 \wedge \ldots \wedge p_l \wedge \neg p_{l +1} \wedge \ldots
\wedge
\neg p_n) &\to \neg f(p_1,\ldots, p_1);\label{E:bool5}\\
(p_1 \wedge \ldots \wedge p_{k' -1} \wedge \neg p_{k'} \wedge
\ldots \wedge \neg p_n) &\to \neg f(p_1,\ldots,
p_1).\label{E:bool6}\\
(p_1 \wedge \ldots \wedge p_{l'} \wedge \neg p_{l' +1} \wedge
\ldots \wedge \neg p_n) &\to f(p_1,\ldots, p_1);\label{E:bool7}
\end{align}
We get then the equivalencies required by Lemma setting as
follows:
$$
B_1, \ldots, B_{k'-1} := \top;
$$
$$
B_{k'}, \ldots, B_{l'} := p_1;
$$
$$
B_{l' +1}, \ldots, B_n := \bot;
$$
$$
C_1, \ldots, C_{k -1} := \top;
$$
$$
C_k, \ldots, C_l := p_1;
$$
$$
C_{l +1}, \ldots, C_n := \bot.
$$

As for Part $8$, assume that $f$ is a non-constant non-$FTF$
Boolean function. Then consider the $n$-tuples of values of
$\bar{p}_n$ for which $f$ is true ($T$-tuples for $f$). Since $f$
is non-$FTF$, then no such tuple can be both below and above some
tuple for which $n$ is false (some $F$-tuple for $f$) in terms of
our induced order among tuples of Boolean values. Therefore we
divide all the $T$-tuples for $f$ into two subsets: (1) lower
$T$-tuples, which are below some $F$-tuple but not above any
$F$-tuple; and (2) upper $T$-tuples, which are above some
$F$-tuple but not below any $F$-tuple. We now set $k$ to be the
number of minimal upper $T$-tuples for $f$ and and $m$ to be the
number of maximal lower $T$-tuples for $f$. Since $f$ is
non-constant, we will obviously have both some $T$-tuples and some
$F$-tuples for $f$, therefore we get that $m + k > 0$, even though
$m$ or $k$ can turn out to be $0$. We now set for $g \leq k$ that
$p^g_1,\ldots, p^g_{i_g}$ is the set of all variables from
$\bar{p}_n$ which are true in the $g$-th minimal upper $T$-tuple,
while for $h \leq m$ we set that $q^h_1,\ldots, q^h_{j_h}$ is the
set of all variables from $\bar{p}_n$ which are false in the
$h$-th maximal lower $T$-tuple.

One straightforwardly verifies that these settings give us the
desired representation for $f$.

Part $9$ is just a dual of Part $8$.
\end{proof}

Now, a guarded $x$-fragment $\mathcal{L}^\Theta_x(\mathbb{M})$ we
call a \emph{standard $x$-fragment} iff $\mathbb{M}$ is a finite
set of flat connectives plus some regular connectives of degree
$2$. The guarded connectives used to generate standard fragments
of the correspondence language we will also call standard
connectives.

It is easy to see, that every $x$-g.c. listed in Example $1$ above
except for $\lambda_3$ is a binary flat connective, and that every
$x$-g.c. from the same Example except for $\lambda_5$ is a
modality. Therefore, given the degrees of these $x$-g.c.'s, every
guarded $x$-fragment of the correspondence language listed in
Example $3$ is a standard $x$-fragment.

\section{Asimulations}\label{S:Asim}

In order to define asimulations we first need to define some
special classes of tuples of binary relations. So let $M_1, M_2$
be $\Theta$-models. We then denote the set of binary relations
satisfying condition \eqref{E:c-typ} for the given $M_1, M_2$ by
$W(M_1, M_2)$. Further, for a set $\Gamma(x)$ of $\Theta$-formulas
we define $Rel(\Gamma(x), M_1, M_2)$ to be the set of all $A \in
W(M_1, M_2)$, such that $\Gamma(x)$ is invariant w.r.t. $A$.

First, we need to handle the propositional cores of guarded
connectives. We bring Boolean functions into correspondence with
the above defined operations on sets of the form $\beta \subseteq
W(M_1, M_2)$ in the following way:
\begin{align*}
 &[f(p_1,\ldots, p_n)](\beta) = \begin{cases}
W(M_1, M_2) & \mbox{if $f$ is constant};\\
\beta & \mbox{if $f$ is non-constant monotone};\\
\beta^{-1} & \mbox{if $f$ is non-constant anti-monotone};\\
\beta\cap\beta^{-1} & \mbox{otherwise.}
\end{cases}
\end{align*}
In the above definition, we assume that:
$$
\beta^{-1} = \{ R^{-1} \mid R \in \beta \},
$$

and that:
$$
\beta \cap \beta^{-1} = \{ R \cap R^{-1} \mid R \in \beta \}.
$$

 Now we can define operations of the form $[\mu](\beta)$,
 where $\mu$ is a binary $x$-g.c. and $\beta \subseteq
W(M_1, M_2)$. These operations are defined for subsets of $W(M_1,
M_2)$, and return the subsets of $W^{n + 1}(M_1, M_2)$, where $n =
\delta(\mu)$. We define the operations by induction on the degree
of $x$-g.c. $\mu$.

\emph{Basis}. If $\delta(\mu) = 0$, then we stipulate that

$$
[\mu](\beta) = [f](\beta),
$$

where  $f$ is the binary Boolean function defined by
 $\psi(P_1(x),\ldots, P_n(x))$ for $P_1(x),\ldots, P_n(x)$.

\emph{Induction step}. If $\delta(\mu) = n + 1$ and $\mu$ is
$\forall$-guarded, then we distinguish between two cases:

\emph{Case 1}. If $\mu$ is not special, then we have $ \mu =
\forall x_2\ldots x_{m + 1}(\pi^{m}_1Sx \to \mu^-)$ in the
assumptions of \eqref{E:Aguard}. We define $[\mu](\beta)$ as the
set of tuples of the form $\langle A_1,\ldots, A_{\delta(\mu) +
1}\rangle$, such that $\langle A_1,\ldots, A_{\delta(\mu)}\rangle
\in [\mu^-](\beta)$ and $A_{\delta(\mu) + 1}$ satisfies the
following condition for all natural $r,t$ such that $\{ r,t \} =
\{ 1,2 \}$:
\begin{align}
&(\forall a_1 \in U_r)(\forall \bar{b}_{m+1} \in
U_t)(a_1\mathrel{A_{\delta(\mu) + 1}}b_1 \wedge \pi^m_1S^tb) \Rightarrow\notag\\
&\qquad\qquad\qquad \Rightarrow\exists a_2\ldots a_{m+1} \in
U_r(\pi^m_1S^ra \wedge
a_{m+1}\mathrel{A_{\delta(\mu)}}b_{m+1})\label{E:cond-A}\tag{\text{back}}
\end{align}

\emph{Case 2}. If $\mu$ is special, then we define $[\mu](\beta)$
as the set of couples of the form $\langle B, A\rangle$, such that
$B \in \beta$ and $A \in W(M_1,M_2)$ satisfies the following
condition for all natural $r,t$ such that $\{ r,t \} = \{ 1,2 \}$:
\begin{align}
&(\forall a_1 \in U_r)(\forall \bar{b}_{m+1} \in
U_t)(a_1\mathrel{A}b_1 \wedge \pi^m_1S^tb) \Rightarrow\notag\\
&\qquad\quad \Rightarrow (\exists a_2\ldots a_{m+1} \in
U_r(\pi^m_1S^ra \wedge a_{m+1}\mathrel{B}b_{m+1})
\wedge\notag\\
&\qquad\qquad\quad\wedge\exists c_2\ldots c_{m+1} \in
U_r(\iota_r(S_1)(a_1, c_2) \wedge \pi^m_2S^rc \wedge
c_{m+1}\mathrel{B^{-1}}b_{m+1}))
\label{E:cond-Aspec}\tag{\text{s-back}}
\end{align}

Finally, if $\delta(\mu) = n + 1$ and $\mu$ is $\exists$-guarded,
we again have two cases.

\emph{Case 3}. If $\mu$ is not special, then we have $\mu =
\exists x_2\ldots x_{m + 1}(\pi^{m}_1Sx \wedge \mu^-)$ in the
assumptions of \eqref{E:Eguard}. We define $[\mu](\beta)$ as the
set of tuples of the form $\langle A_1,\ldots, A_{\delta(\mu) +
1}\rangle$, such that $\langle A_1,\ldots, A_{\delta(\mu)}\rangle
\in [\mu^-](\beta)$ and $A_{\delta(\mu) + 1}$ satisfies the
following condition for all natural $r,t$ such that $\{ r,t \} =
\{ 1,2 \}$:
\begin{align}
&(\forall \bar{a}_{m+1} \in U_r)(\forall b_1 \in
U_t)(a_1\mathrel{A_{\delta(\mu) + 1}}b_1 \wedge \pi^m_1S^ra) \Rightarrow\notag\\
&\qquad\qquad\qquad \Rightarrow\exists b_2\ldots b_{m+1} \in
U_t(\pi^m_1S^tb \wedge
a_{m+1}\mathrel{A_{\delta(\mu)}}b_{m+1})\label{E:cond-E}\tag{\text{forth}}
\end{align}

\emph{Case 4}. If $\mu$ is special, then we define $[\mu](\beta)$
as the set of couples of the form $\langle B, A\rangle$, such that
$B \in \beta$ and $A \in W(M_1,M_2)$ satisfies the following
condition for all natural $r,t$ such that $\{ r,t \} = \{ 1,2 \}$:
\begin{align}
&(\forall \bar{a}_{m+1} \in U_r)(\forall b_1 \in
U_t)(a_1\mathrel{A}b_1 \wedge \pi^m_1S^ra) \Rightarrow\notag\\
&\qquad\quad \Rightarrow (\exists b_2\ldots b_{m+1} \in
U_t(\pi^m_1S^tb \wedge a_{m+1}\mathrel{B}b_{m+1})
\wedge\notag\\
&\qquad\qquad\quad\wedge\exists c_2\ldots c_{m+1} \in
U_t(\iota_t(S_1)(b_1, c_2) \wedge \pi^m_2S^tc \wedge
a_{m+1}\mathrel{B^{-1}}c_{m+1}))
\label{E:cond-Espec}\tag{\text{s-forth}}
\end{align}

We now prove an important lemma about the defined operations:

\begin{lemma}\label{L:mu-op}
Let $\varphi(x)$ be logically equivalent to $\mu(\psi_1,\ldots,
\psi_n)$, where $\mu$ is an arbitrary $x$-g.c. Then, if $\beta
\subseteq Rel(\{ \psi_1,\ldots, \psi_n \}, M_1, M_2)$ for some
models $M_1$ and $M_2$, then $\varphi(x)$ is invariant w.r.t. to
the set
$$
\{ A_{\delta(\mu) + 1} \mid \exists A_1,\ldots,
A_{\delta(\mu)}(\langle A_1,\ldots, A_{\delta(\mu) + 1}\rangle \in
[\mu](\beta)) \}.
$$
\end{lemma}
\begin{proof}
We argue by induction on $\delta(\mu)$.

\emph{Basis}. Assume $\delta(\mu) = 0$. Then  $\varphi(x)$ is
logically equivalent to $\psi(\psi_1(x),\ldots, \psi_n(x))$, where
$\psi$ induces some Boolean $f$ on $\psi_1(x),\ldots, \psi_n(x)$.
Therefore, we need to show that $\varphi(x)$ is invariant w.r.t.
every relation in $[\mu](\beta)$. We then have to distinguish
between the following $4$ cases:

\emph{Case 1}. $f$ is constant. Then $[\mu](\beta) = W(M_1,M_2)$
and $\varphi(x)$ is either equivalent to $\top$ or to $\bot$.
Obviously, $\varphi(x)$ is invariant w.r.t. any relation in
$W(M_1,M_2)$

\emph{Case 2}. $f$ is non-constant monotone. Then $[\mu](\beta) =
\beta$. Since $\psi_1,\ldots, \psi_n$ are invariant w.r.t.
$\beta$, it is obvious that $\varphi(x)$ is invariant w.r.t.
$[\mu](\beta) = \beta$ as well.

\emph{Case 3}. $f$ is non-constant anti-monotone. Then
$[\mu](\beta) = \beta^{-1}$. Since, by the assumption of the
lemma, $\psi_1,\ldots, \psi_n$ are invariant w.r.t. to $\beta$,
then, by contraposition, $\neg\psi_1,\ldots, \neg\psi_n$ are
invariant w.r.t. every inverse of a relation from $\beta$.
Therefore, $\varphi(x)$ is invariant w.r.t. $[\mu](\beta) =
\beta^{-1}$.

\emph{Case 4}. $f$ is a rest function, that is to say, neither
monotone nor anti-monotone. Then $[\mu](\beta) =
\beta\cap\beta^{-1}$. If $A \in \beta\cap\beta^{-1}$, then by
assumption of the lemma and contraposition we have both that
$\psi_1,\ldots, \psi_n$ are invariant w.r.t. $A$ and
$\neg\psi_1,\ldots, \neg\psi_n$ are invariant w.r.t. $A$.
Therefore every Boolean combination of $\psi_1,\ldots, \psi_n$ is
also invariant w.r.t. $A$. Since $A$ was chosen arbitrarily, this
means that $\varphi(x)$ is invariant w.r.t. $[\mu](\beta)$.

\emph{Induction step}. Assume $\delta(\mu) = k + 1$. We then have
to distinguish between the following $4$ cases:

\emph{Case 1}. If $\mu$ is $\forall$-guarded and not special, then
$\mu(\psi_1,\ldots, \psi_n)$ and thus $\varphi(x)$ is equivalent
to a formula $\forall x_2\ldots x_{m + 1}(\pi^{m}_1Sx \to \mu^-)$
under the assumptions of \eqref{E:Aguard}. Assume that $\langle
A_1,\ldots, A_{\delta(\mu) + 1}\rangle$ is in $[\mu](\beta)$. Then
we also have that $\langle A_1,\ldots, A_{\delta(\mu)}\rangle$ is
in $[\mu^-](\beta)$.

Now, assume that $\{ r,t \} = \{ 1,2 \}$, $a_1 \in U_r$, $b_1 \in
U_t$, and $a_1\mathrel{A_{\delta(\mu) + 1}}b_1$. Moreover, assume
that

\begin{equation}\label{E:1}
a_1 \models_r \mu(\psi_1,\ldots, \psi_n).
\end{equation}

Then let $b_2\ldots b_{m+1} \in U_t$ be such that

\begin{equation}\label{E:2}
   \pi^m_1S^tb.
\end{equation}
Then, by condition \eqref{E:cond-A} there exist $a_2\ldots a_{m+1}
\in U_r$ such that:
\begin{equation}\label{E:3}
    \pi^m_1S^ra,
\end{equation}
and
\begin{equation}\label{E:4}
a_{m+1}\mathrel{A_{\delta(\mu)}}b_{m+1}.
\end{equation}
By \eqref{E:1} and \eqref{E:3} we know that
\begin{equation}\label{E:5}
a_{m+1} \models_r \mu^-(\psi_1,\ldots, \psi_n).
\end{equation}
By \eqref{E:4}, the fact that $\langle A_1,\ldots,
A_{\delta(\mu)}\rangle \in [\mu^-](\beta)$ and induction
hypothesis, we have
\begin{equation}\label{E:6}
b_{m+1} \models_t \mu^-(\psi_1,\ldots, \psi_n).
\end{equation}
Since  $b_2\ldots b_{m+1}$ were chosen arbitrarily under the
condition \eqref{E:2}, we get that $b_1 \models_t
\mu(\psi_1,\ldots, \psi_n)$, and thus we are done.

\emph{Case 2}. Assume that $\mu$ is $\forall$-guarded and special.
Then we have $\mu \in \nu_x(\forall, f)$, and $f$, being
$\forall$-special, is not a $TFT$-function. Also,
$\mu(\psi_1,\ldots, \psi_n)$ is equivalent to

$$
\forall x_2\ldots x_{m + 1}(\pi^{m}_1Sx \to \psi(\psi_1(x_{m + 1})
,\ldots, \psi_n(x_{m + 1}))),
$$
where $\psi$ induces $f$ for $\psi_1(x_{m + 1}) ,\ldots,
\psi_n(x_{m + 1})$.

Let $\langle B,A\rangle \in [\mu](\beta)$ and assume that $\{ r,t
\} = \{ 1,2 \}$, $a_1 \in U_r$, $b_1 \in U_t$, and
$a_1\mathrel{A}b_1$. Moreover, assume that

\begin{equation}\label{E:1a}
a_1 \models_r \mu(\psi_1,\ldots, \psi_n).
\end{equation}

Then let $b_2\ldots b_{m+1} \in U_t$ be such that

\begin{equation}\label{E:2a}
   \pi^m_1S^tb.
\end{equation}
Then, by condition \eqref{E:cond-Aspec} there exist $a_2\ldots
a_{m+1} \in U_r$ such that:
\begin{equation}\label{E:3a}
    \pi^m_1S^ra,
\end{equation}
and
\begin{equation}\label{E:4a}
a_{m+1}\mathrel{B}b_{m+1}.
\end{equation}
Moreover, by the same condition there exist $c_2\ldots c_{m+1} \in
U_r$ such that:
\begin{equation}\label{E:3aa}
    \iota_r(S_1)(a_1, c_2) \wedge \pi^m_2S^rc
\end{equation}
and
\begin{equation}\label{E:4aa}
c_{m+1}\mathrel{B^{-1}}b_{m+1}.
\end{equation}

 By \eqref{E:1a} and \eqref{E:3a} we
know that
\begin{equation}\label{E:5a}
a_{m+1} \models_r \psi(\psi_1(x_{m + 1}) ,\ldots, \psi_n(x_{m +
1})),
\end{equation}
that is to say, that the $n$-tuple $\bar{\alpha}_n$ of Boolean
values induced by $\psi_1(x_{m + 1}) ,\ldots, \psi_n(x_{m + 1})$
on $(M_r, a_{m+1})$, verifies function $f$.

Further, by \eqref{E:1a} and \eqref{E:3aa} we know that
\begin{equation}\label{E:5aa}
c_{m+1} \models_r \psi(\psi_1(x_{m + 1}) ,\ldots, \psi_n(x_{m +
1})),
\end{equation}
that is to say, that the $n$-tuple $\bar{\gamma}_n$ of Boolean
values induced by $\psi_1(x_{m + 1}) ,\ldots, \psi_n(x_{m + 1})$
on $(M_r, c_{m+1})$, verifies function $f$.

Now, let $\bar{\eta}_n$ be the $n$-tuple of Boolean values induced
by $\psi_1(x_{m + 1}) ,\ldots, \psi_n(x_{m + 1})$ on $(M_t,
b_{m+1})$. By \eqref{E:4a}, \eqref{E:4aa}, and the fact that $B
\in \beta \subseteq Rel(\{ \psi_1,\ldots, \psi_n \}, M_1, M_2)$,
we know that

$$
\bar{\alpha}_n \leq \bar{\eta}_n \leq \bar{\gamma}_n.
$$

Therefore, by \eqref{E:5a}, \eqref{E:5aa}, and the fact that $f$
is not a $TFT$-function, we must have that:
\begin{equation}\label{E:6a}
b_{m+1} \models_t \psi(\psi_1(x_{m + 1}) ,\ldots, \psi_n(x_{m +
1})).
\end{equation}

Since  $b_2\ldots b_{m+1}$ were chosen arbitrarily under the
condition \eqref{E:2a}, we get that

\noindent$b_1 \models_t \mu(\psi_1,\ldots, \psi_n)$, and thus we
are done.

\emph{Case 3}. If $\mu$ is $\exists$-guarded and not special, then
$\mu(\psi_1,\ldots, \psi_n)$ and thus $\varphi(x)$ is equivalent
to a formula $\exists x_2\ldots x_{m + 1}(\pi^{m}_1Sx \wedge
\mu^-)$ under the assumptions of \eqref{E:Eguard}. Assume that
$\langle A_1,\ldots, A_{\delta(\mu) + 1}\rangle$ is in
$[\mu](\beta)$. Then we also have that $\langle A_1,\ldots,
A_{\delta(\mu)}\rangle$ is in $[\mu^-](\beta)$.

Now, assume that $\{ r,t \} = \{ 1,2 \}$, $a_1 \in U_r$, $b_1 \in
U_t$, and $a_1\mathrel{A_{\delta(\mu) + 1}}b_1$. Moreover, assume
that

\begin{equation}\label{E:1c}
a_1 \models_r \mu(\psi_1,\ldots, \psi_n).
\end{equation}

Then let $a_2\ldots a_{m+1} \in U_r$ be such that

\begin{equation}\label{E:2c}
   \pi^m_1S^ra,
\end{equation}
and
\begin{equation}\label{E:3c}
a_{m+1} \models_r \mu^-(\psi_1,\ldots, \psi_n).
\end{equation}
Then, by condition \eqref{E:cond-E}, there exist $b_2\ldots
b_{m+1} \in U_t$ such that
\begin{equation}\label{E:4c}
    \pi^m_1S^tb,
\end{equation}
and
\begin{equation}\label{E:5c}
a_{m+1}\mathrel{A_{\delta(\mu)}}b_{m+1}.
\end{equation}

By \eqref{E:5c}, the fact that $\langle A_1,\ldots,
A_{\delta(\mu)}\rangle \in [\mu^-](\beta)$ and induction
hypothesis, we have
\begin{equation}\label{E:6c}
b_{m+1} \models_t \mu^-(\psi_1,\ldots, \psi_n).
\end{equation}
Therefore, by \eqref{E:4c} and \eqref{E:6c}, we get that $b_1
\models_t \mu(\psi_1,\ldots, \psi_n)$, and thus we are done.

\emph{Case 4}. Assume that $\mu$ is $\exists$-guarded and special.
Then we have $\mu \in \nu_x(\exists, f)$, and $f$, being
$\exists$-special, is not an $FTF$-function. Also,
$\mu(\psi_1,\ldots, \psi_n)$ is equivalent to

$$
\exists x_2\ldots x_{m + 1}(\pi^{m}_1Sx \wedge \psi(\psi_1(x_{m +
1}) ,\ldots, \psi_n(x_{m + 1}))),
$$
where $\psi$ induces $f$ for $\psi_1(x_{m + 1}) ,\ldots,
\psi_n(x_{m + 1})$.

Let $\langle B,A\rangle \in [\mu](\beta)$ and assume that $\{ r,t
\} = \{ 1,2 \}$, $a_1 \in U_r$, $b_1 \in U_t$, and
$a_1\mathrel{A}b_1$. Moreover, assume that

\begin{equation}\label{E:1d}
a_1 \models_r \mu(\psi_1,\ldots, \psi_n).
\end{equation}

Then one can choose $a_2\ldots a_{m+1} \in U_r$ such that

\begin{equation}\label{E:2d}
   \pi^m_1S^ra,
\end{equation}
and
\begin{equation}\label{E:3d}
  a_{m+1} \models_r \psi(\psi_1(x_{m + 1}) ,\ldots, \psi_n(x_{m +
1})).
\end{equation}
Therefore, we know that the $n$-tuple $\bar{\eta}_n$ of Boolean
values induced by $\psi_1(x_{m + 1}) ,\ldots, \psi_n(x_{m + 1})$
on $(M_r, a_{m+1})$, verifies function $f$.

But then, by condition \eqref{E:cond-Espec} there exist $b_2\ldots
b_{m+1} \in U_t$ such that:
\begin{equation}\label{E:4d}
    \pi^m_1S^tb,
\end{equation}
and
\begin{equation}\label{E:5d}
a_{m+1}\mathrel{B}b_{m+1}.
\end{equation}
Moreover, by the same condition there exist $c_2\ldots c_{m+1} \in
U_t$ such that:
\begin{equation}\label{E:4dd}
    \iota_t(S_1)(b_1, c_2) \wedge \pi^m_2S^tc
\end{equation}
and
\begin{equation}\label{E:5dd}
a_{m+1}\mathrel{B^{-1}}c_{m+1}.
\end{equation}
Now, let $\bar{\alpha}_n$ and $\bar{\beta}_n$ be the $n$-tuples of
Boolean values induced by $\psi_1(x_{m + 1}) ,\ldots, \psi_n(x_{m
+ 1})$ on $(M_t, b_{m+1})$ and $(M_t, c_{m+1})$ respectively. By
\eqref{E:5d}, \eqref{E:5dd}, and the fact that

\noindent$B \in \beta \subseteq Rel(\{ \psi_1,\ldots, \psi_n \},
M_1, M_2)$, we know that

$$
\bar{\beta}_n \leq \bar{\eta}_n \leq \bar{\alpha}_n.
$$
Therefore, since $f$ is not an $FTF$-function, we know that at
least one of the tuples $\bar{\alpha}_n$, $\bar{\beta}_n$ verifies
$f$. Whence we have either

\begin{equation}\label{E:6d}
b_{m+1} \models_t \psi(\psi_1(x_{m + 1}) ,\ldots, \psi_n(x_{m +
1})),
\end{equation}
or

\begin{equation}\label{E:7d}
c_{m+1} \models_t \psi(\psi_1(x_{m + 1}) ,\ldots, \psi_n(x_{m +
1})).
\end{equation}
In both cases, using either \eqref{E:4d} or \eqref{E:4dd}, we get
that $b_1 \models_t \mu(\psi_1,\ldots, \psi_n)$, and thus we are
done.
\end{proof}

We are now ready to define asimulations, the central notion of
this paper.

\begin{definition}\label{D:asim}
Let $\mathcal{L}^\Theta_x(\mu_1,\ldots, \mu_s)$ be a guarded
$x$-fragment of the correspondence language and let $M_1, M_2$ be
$\Theta$-models. A non-empty relation

\noindent$A \in Rel(\mathcal{L}^\Theta_x(\emptyset),M_1, M_2)$ is
an \emph{$(\mathcal{L}^\Theta_x(\mathbb{M}), M_1,
M_2)$-asimulation} iff for every $i$ such that $1 \leq i \leq s$
there exist $A_1,\ldots, A_{\delta(\mu_i)}$ such that
$$
\langle A_1,\ldots, A_{\delta(\mu_i)}, A\rangle \in [\mu](\{ A
\}).
$$
\end{definition}

The fact that $A$ is an $(\mathcal{L}^\Theta_x(\mathbb{M}), M_1,
M_2)$-asimulation we will abbreviate by

\noindent$A \in A\sigma(\mathcal{L}^\Theta_x(\mathbb{M}), M_1,
M_2)$.

\begin{example}
{\em

In the notation of the two examples given in the previous section,
we get as a result of above definition, that for any two given
models $M_1$ and $M_2$:

\begin{enumerate}
\item The set $A\sigma(\mathcal{L}^\Sigma_x(\wedge, \vee, \bot,
\top, \lambda_5), M_1, M_2)$ is the set of all asimulations
between $M_1$ and $M_2$ as defined in \cite{Ol13}.

\item The set $A\sigma(\mathcal{L}^\Sigma_x(\wedge, \vee, \bot,
\top, \neg, \lambda_1, \lambda_4), M_1, M_2)$ is the set of all
bisimulations between $M_1$ and $M_2$.

\item The set $A\sigma(\mathcal{L}^\Sigma_x(\wedge, \vee, \bot,
\top, \lambda_2, \lambda_3, \lambda_5), M_1, M_2)$ is the set of
all relations $A$ for which there exists a relation $B$ such that
$\langle A, B\rangle$ is a $(2,2)$-modal asimulation between $M_1$
and $M_2$ as defined in \cite[Definition 5]{Ol15}.
\end{enumerate}
}
\end{example}

With the help of Definition \ref{D:asim}, we obtain one part of
our characterization as a corollary of Lemma \ref{L:mu-op}:

\begin{corollary}\label{L:cor}
Let $\mathcal{L}^\Theta_x(\mu_1,\ldots, \mu_s)$ be a guarded
$x$-fragment of the correspondence language and let $M_1, M_2$ be
$\Theta$-models. If $A \in
A\sigma(\mathcal{L}^\Theta_x(\mathbb{M}), M_1, M_2)$ and
$\varphi(x)$ is logically equivalent to a formula in
$\mathcal{L}^\Theta_x(\mathbb{M})$, then $\varphi(x)$ is invariant
w.r.t. $A$.
\end{corollary}

\begin{proof}
Assume that $\varphi(x)$ is equivalent to a formula in $\psi(x)
\in \mathcal{L}^\Theta_x(\mathbb{M})$. We argue by induction on
the construction of $\psi(x)$. If, for some $P_n \in \Theta$,
$\psi(x)$ is $P_n(x)$, then $\psi(x)$ (and therefore $\varphi(x)$)
is invariant w.r.t. $A$, since $A \in
Rel(\mathcal{L}^\Theta_x(\emptyset),M_1, M_2)$.

If, for some $i$, such that $1 \leq i \leq s$, $\psi(x)$ is of the
form $\mu_i(\chi_1,\ldots, \chi_n)$, then there exist $A_1,\ldots,
A_{\delta(\mu_i)}$ such that
$$
\langle A_1,\ldots, A_{\delta(\mu_i)}, A\rangle \in [\mu](\{ A
\}).
$$
By induction hypothesis, we know that $A \in Rel(\{ \chi_1,\ldots,
\chi_n \}, M_1, M_2)$. Therefore, setting $\beta := \{ A \}$ in
Lemma \ref{L:mu-op}, we get that $\varphi(x)$, being equivalent to
$\mu_i(\chi_1, ,\ldots, \chi_n)$, is invariant w.r.t. $A$.
\end{proof}

We note that Corollary \ref{L:cor} applies to arbitrary guarded
fragments rather than to just standard ones and is therefore much
stronger than the `easy', left-to-right direction of our main
result, Theorem \ref{L:final}.

\section{Asimulations over saturated models}\label{S:Sat}

To proceed, we need to introduce some further notions and results
from classical model theory. For a $\Theta$-model $M$ and
$\bar{a}_n \in U$, let $M/\bar{a}_n$ be the extension of $M$ with
$\bar{a}_n$ as new individual constants interpreted as themselves.
It is easy to see that there is a simple relation between the
truth of a formula at a sequence of elements of a $\Theta$-model
and the truth of its substitution instance in an extension of the
above-mentioned kind; namely, for any $\Theta$-model $M$, any
$\Theta$-formula $\varphi(\bar{y}_n,\bar{w}_m)$ and any
$\bar{a}_n,\bar{b}_m \in U$ it is true that:

\[
(M/\bar{a}_n), \bar{b}_m \models \varphi(\bar{a}_n,\bar{w}_m)
\Leftrightarrow M, \bar{a}_n, \bar{b}_m \models
\varphi(\bar{y}_n,\bar{w}_m).
\]

We will call a theory of $M$ (and write $Th(M)$) the set of all
first-order sentences true at $M$. We will call an $n$-type of $M$
a set of formulas $\Gamma(\bar{w}_n)$ consistent with $Th(M)$.

\begin{definition}
Let $M$ be a $\Theta$-model. $M$ is \emph{$\omega$-saturated} iff
for all $k \in \mathbb{N}$ and for all $\bar{a}_n \in U$, every
$k$-type $\Gamma(\bar{w}_k)$ of $M/\bar{a}_n$ is satisfiable in
$M/\bar{a}_n$.
\end{definition}

Definition of $\omega$-saturation normally requires satisfiability
of $1$-types only. However, our modification is equivalent to the
more familiar version: see e.g. \cite[Lemma 4.31, p. 73]{Doets96}.

It is known that every model can be elementarily extended to an
$\omega$-saturated model; in other words, the following lemma
holds:

\begin{lemma}\label{L:ext}
Let $M$ be a $\Theta$-model. Then there is an $\omega$-saturated
extension $M'$ of $M$ such that for all $\bar{a}_n \in U$ and
every $\Theta$-formula $\varphi(\bar{w}_n)$:
\[
M, \bar{a}_n \models \varphi(\bar{w}_n) \Leftrightarrow M',
\bar{a}_n \models \varphi(\bar{w}_n).
\]
\end{lemma}
The latter lemma is a trivial corollary of e.g. \cite[Lemma
5.1.14, p. 216]{ChK73}.

In what follows, some types will be of special interest to us. If
$\Gamma(x)$ is a set of formulas, $M$ is a model and $a \in U$,
then we can define two further sets of formulas on the basis of
$\Gamma$:
$$
Tr(\Gamma(x), M, a) = \{ \psi(x) \in \Gamma \mid M, a \models
\psi(x) \},
$$
and
$$
Fa(\Gamma(x), M, a) = \{ \psi(x) \in \Gamma \mid M, a \not\models
\psi(x) \}.
$$

Saturated models are convenient since they allow to define
asimulations over them in a rather straightforward way. But before
we approach this feature of saturated models, we need to collect
some technical facts about modalities of degree $1$:

\begin{lemma}\label{L:modalities}
Let $\mu$ be an $x$-modality of degree $1$ and let
$\mathcal{L}^\Theta_x(\mathbb{M})$ be a guarded $x$-fragment of
the correspondence language, such that $\mathbb{M} \supseteq \{
\wedge, \vee, \top, \bot \}$. Let $\psi_1,\ldots, \psi_u$ be
arbitrary formulas in the correspondence language. Then:
\begin{enumerate}
\item $\{ \mu(\psi_1,\ldots, \psi_n)\mid \psi_1,\ldots, \psi_n \in
\mathcal{L}^\Theta_x(\mathbb{M}) \} = \{ \mu(\psi_1,\ldots,
\psi_1)\mid \psi_1 \in \mathcal{L}^\Theta_x(\mathbb{M}) \}$.

\item If $\mu$ is $\forall$-guarded and has a monotone
propositional core, then
$$
\models  \bigwedge^u_{i = 1}\mu(\psi_i,\ldots, \psi_i)
\leftrightarrow \mu(\bigwedge^u_{i = 1}\psi_i,\ldots,
\bigwedge^u_{i = 1}\psi_i).
$$

\item If $\mu$ is $\forall$-guarded and has an anti-monotone
propositional core, then
$$
\models  \bigwedge^u_{i = 1}\mu(\psi_i,\ldots, \psi_i)
\leftrightarrow \mu(\bigvee^u_{i = 1}\psi_i,\ldots, \bigvee^u_{i =
1}\psi_i).
$$

\item If $\mu$ is $\exists$-guarded and has a monotone
propositional core, then
$$
\models  \bigvee^u_{i = 1}\mu(\psi_i,\ldots, \psi_i)
\leftrightarrow \mu(\bigvee^u_{i = 1}\psi_i,\ldots, \bigvee^u_{i =
1}\psi_i).
$$

\item If $\mu$ is $\exists$-guarded and has an anti-monotone
propositional core, then
$$
\models  \bigvee^u_{i = 1}\mu(\psi_i,\ldots, \psi_i)
\leftrightarrow \mu(\bigwedge^u_{i = 1}\psi_i,\ldots,
\bigwedge^u_{i = 1}\psi_i).
$$
\end{enumerate}
\end{lemma}
\begin{proof} (Part $1$) Right-to-left inclusion is obvious. In
the other direction, let $\mu$ be an $x$-modality of degree $1$
and let  $\psi_1,\ldots, \psi_n \in
\mathcal{L}^\Theta_x(\mathbb{M})$. Then $\mu^-$ may define either
a monotone or an anti-monotone Boolean function for $P_1, \ldots,
P_n$. So we have $2$ cases to consider:

\emph{Case 1}. If  $\mu^-$ defines a non-constant monotone Boolean
function for $P_1, \ldots, P_n$, then by Lemma \ref{L:Boolean}.2
we have
$$
\models \mu^-(\psi,\ldots, \psi) \leftrightarrow \psi
$$
for every formula $\psi$ in the correspondence language. On the
other hand, by Lemma \ref{L:Boolean}.1, there is a superposition
$F$ of $\wedge$'s and $\vee$'s such that
$$
\models \mu^-(\psi_1,\ldots, \psi_n) \leftrightarrow
F(\psi_1,\ldots, \psi_n).
$$
Adding the two equivalencies together, we get that

$$
\models \mu^-(\psi_1,\ldots, \psi_n) \leftrightarrow
\mu^-(F(\psi_1,\ldots, \psi_n),\ldots, F(\psi_1,\ldots, \psi_n)),
$$
and therefore, that:
$$
\models \mu(\psi_1,\ldots, \psi_n) \leftrightarrow
\mu(F(\psi_1,\ldots, \psi_n),\ldots, F(\psi_1,\ldots, \psi_n)).
$$
Note, further, that since $\wedge, \vee \in \mathbb{M}$, we also
have $F(\psi_1,\ldots, \psi_n) \in
\mathcal{L}^\Theta_x(\mathbb{M})$, and, therefore:
$$
\mu(F(\psi_1,\ldots, \psi_n),\ldots, F(\psi_1,\ldots, \psi_n)) \in
\{ \mu(\psi,\ldots, \psi)\mid \psi \in
\mathcal{L}^\Theta_x(\mathbb{M}) \}.
$$
Since we
do not distinguish between equivalent formulas, the latter
equivalence proves the left-to-right inclusion.

\emph{Case 2}. If  $\mu^-$ defines an anti-monotone Boolean
function for $P_1, \ldots, P_n$, then by Lemma \ref{L:Boolean}.4
we have
$$
\models \mu^-(\psi,\ldots, \psi) \leftrightarrow \neg\psi
$$
for every formula $\psi$ in the correspondence language. On the
other hand, by Lemma \ref{L:Boolean}.3, there is a superposition
$F$ of $\wedge$'s and $\vee$'s such that
$$
\models \mu^-(\psi_1,\ldots, \psi_n) \leftrightarrow \neg
F(\psi_1,\ldots, \psi_n).
$$
Adding the two equivalencies together, we get that

$$
\models \mu^-(\psi_1,\ldots, \psi_n) \leftrightarrow
\mu^-(F(\psi_1,\ldots, \psi_n),\ldots, F(\psi_1,\ldots, \psi_n)),
$$
and therefore, that:
$$
\models \mu(\psi_1,\ldots, \psi_n) \leftrightarrow
\mu(F(\psi_1,\ldots, \psi_n),\ldots, F(\psi_1,\ldots, \psi_n)).
$$
Note, further, that since $\wedge, \vee \in \mathbb{M}$, we also
have $F(\psi_1,\ldots, \psi_n) \in
\mathcal{L}^\Theta_x(\mathbb{M})$. Since we do not distinguish
between equivalent formulas, then, reasoning as in the previous
case, the latter equivalence proves the left-to-right inclusion.

(Part $2$) We build a chain of logical equivalents connecting both
parts of the biconditional in the statement of this part of the
Lemma. Assume that $\mu$ is $\forall$-guarded, and that $\mu^-$,
its propositional core, defines a monotone Boolean function. Then
$\mu$ has a form
$$
\forall x_2\ldots x_{m + 1}(\pi^{m}_1Sx \to \mu^-)
$$
in the assumptions of \ref{E:Aguard}. By Lemma \ref{L:Boolean}.2
we have
$$
\models \mu^-(\psi_i,\ldots, \psi_i) \leftrightarrow \psi_i
$$
for every $1 \leq i \leq u$; therefore the formula $\bigwedge^u_{i
= 1}\mu(\psi_i,\ldots, \psi_i)$ is logically equivalent to
$$
\bigwedge^u_{i = 1}(\forall x_2\ldots x_{m + 1}(\pi^{m}_1Sx \to
\psi_i)).
$$
Using the distributivity of universal quantifier over conjunction
we get then the following chain of logical equivalents for the
latter formula:
$$
\forall x_2\ldots x_{m + 1}\bigwedge^u_{i = 1}(\pi^{m}_1Sx \to
\psi_i),
$$
$$
\forall x_2\ldots x_{m + 1}(\pi^{m}_1Sx \to \bigwedge^u_{i =
1}\psi_i).
$$
Again, using monotonicity of the function defined by $\mu^-$ and
Lemma \ref{L:Boolean}.2, we proceed in this chain of equivalences
as follows:
$$
\forall x_2\ldots x_{m + 1}(\pi^{m}_1Sx \to \mu^-(\bigwedge^u_{i =
1}\psi_i, \ldots, \bigwedge^u_{i = 1}\psi_i)),
$$
$$
\mu(\bigwedge^u_{i = 1}\psi_i,\ldots, \bigwedge^u_{i = 1}\psi_i),
$$
and thus we are done.

(Part $3$) Again we proceed by building an appropriate chain of
logical equivalents. Assume that $\mu$ is $\forall$-guarded, and
that $\mu^-$, its propositional core, defines an anti-monotone
Boolean function. Then $\mu$ has a form
$$
\forall x_2\ldots x_{m + 1}(\pi^{m}_1Sx \to \mu^-)
$$
in the conditions of \ref{E:Aguard}. By Lemma \ref{L:Boolean}.4 we
have
$$
\models \mu^-(\psi_i,\ldots, \psi_i) \leftrightarrow \neg\psi_i
$$
for every $1 \leq i \leq u$; therefore the formula $\bigwedge^u_{i
= 1}\mu(\psi_i,\ldots, \psi_i)$ is logically equivalent to
$$
\bigwedge^u_{i = 1}(\forall x_2\ldots x_{m + 1}(\pi^{m}_1Sx \to
\neg\psi_i)).
$$
Using the distributivity of universal quantifier over conjunction
we get then the following chain of logical equivalents for the
latter formula:
$$
\forall x_2\ldots x_{m + 1}\bigwedge^u_{i = 1}(\pi^{m}_1Sx \to
\neg\psi_i),
$$
$$
\forall x_2\ldots x_{m + 1}(\pi^{m}_1Sx \to \bigwedge^u_{i =
1}\neg\psi_i),
$$
$$
\forall x_2\ldots x_{m + 1}(\pi^{m}_1Sx \to \neg\bigvee^u_{i =
1}\psi_i),
$$
Again, using anti-monotonicity of the function defined by $\mu^-$
and Lemma \ref{L:Boolean}.4 we proceed in this chain of
equivalences as follows:
$$
\forall x_2\ldots x_{m + 1}(\pi^{m}_1Sx \to \mu^-(\bigvee^u_{i =
1}\psi_i, \ldots, \bigvee^u_{i = 1}\psi_i)),
$$
$$
\mu(\bigvee^u_{i = 1}\psi_i,\ldots, \bigvee^u_{i = 1}\psi_i),
$$
and thus we are done.

Parts $4$ and $5$ of the Lemma are dual to the parts $2$ and $3$
and can be proven by a similar method.
\end{proof}
We turn now to the key Lemma about asimulations over saturated
models:

\begin{lemma}\label{L:mu-sat}
Let $\mathcal{L}^\Theta_x(\mathbb{M})$ be a standard $x$-fragment
of the correspondence language, such that $\mathbb{M} = \{
\mu_1,\ldots, \mu_s \} \supseteq \{ \wedge, \vee, \top, \bot \}$,
let $M_1, M_2$ be $\omega$-saturated $\Theta$-models, and let $A
\in W(M_1, M_2)$ be such that
$$
A = \bigcup Rel(\mathcal{L}^\Theta_x(\mathbb{M}),M_1, M_2).
$$

Further, assume that $\mu$ is a standard $x$-g.c. (not necessarily
in $\mathbb{M}$) such that $\delta(\mu) \geq 1$, $\mu^0,\ldots,
\mu^{\delta(\mu) - 1}$ is the set of ancestors of $\mu$, and for
$1 \leq i \leq \delta(\mu)$ define $A_i \in W(M_1, M_2)$, such
that
$$
A_i = \bigcup Rel(\{ \mu^{i-1}(\psi_1,\ldots, \psi_n)\mid
\psi_1,\ldots, \psi_n \in \mathcal{L}^\Theta_x(\mathbb{M}) \},M_1,
M_2).
$$
Then for any $B \in W(M_1, M_2)$, such that
$$
B \in Rel(\{ \mu(\psi_1,\ldots, \psi_n)\mid \psi_1,\ldots, \psi_n
\in \mathcal{L}^\Theta_x(\mathbb{M}) \},M_1, M_2),
$$
and for some $C \in \{ A \} \cup [\mu^0](\{ A \})$, the tuple
$\langle C,A_2,\ldots, A_{\delta(\mu)}, B\rangle$ is in $[\mu](\{
A \})$.
\end{lemma}
\begin{proof}

We need to distinguish between $2$ sets of cases corresponding to
the two kinds of standard connectives, that is to say, arbitrary
guarded connectives of degree $1$ and regular connectives of
degree $2$ respectively. To reduce the length of formulas below,
we introduce the following abbreviations for arbitrary variable
$y$, natural $j \in \{ 1,2 \}$, and $a \in U_j$:
$$
Tr^j_y(a):= Tr(\mathcal{L}^\Theta_y(\mathbb{M}), M_j, a),
$$
and
$$
Fa^j_y(a):= Fa(\mathcal{L}^\Theta_y(\mathbb{M}), M_j, a).
$$
Further, we set that:
$$
\neg Fa^j_y(a) := \{ \neg\psi(x) \mid \psi(x) \in Fa^j_y(a) \}.
$$

\emph{Case 1}. Let $\mu$ be a flat connective. Then one of the
following cases holds:

\emph{Case 1.1}. Assume that $\mu$ has a constant propositional
core. Then we have $\mu \in \nu_x(\forall, \bot) \cup
\nu_x(\exists, \top)$ and also that $\{ A \} \cup [\mu^0](\{ A \})
= W(M_1,M_2)$. We will show that for the binary relation
$$
C = \bigcup W(M_1,M_2) = (U_1\times U_2) \cup (U_2\times U_1) \in
W(M_1,M_2)
$$
and for every $B$, satisfying the lemma hypothesis, we have
$\langle C, B\rangle \in [\mu](\{ A \})$.

Thus, assume that $\mu \in \nu_x(\forall, \bot)$, so that $\mu$
has the form $\forall x_2\ldots x_{m + 1}(\pi^{m}_1Sx \to \bot)$
in the assumptions of \eqref{E:Aguard}. Assume that $\{ r,t \} =
\{ 1,2 \}$ and that we have $a_1 \in U_r$, $\bar{b}_{m+1} \in
U_t$, and $a_1\mathrel{B}b_1$. Moreover, assume that
$\pi^m_1S^tb$. Then we have $b_1 \not\models_t \mu$ and hence, by
$a_1\mathrel{B}b_1$, that $a_1 \not\models_r \mu$. But the latter
means that the formula $\pi^m_1Sx$ is satisfiable at $(M_r, a_1)$
by some $a_2,\ldots, a_{m+1}  \in U_r$. So for any such
$a_2,\ldots, a_{m+1}$ we will have both $\pi^m_1S^ra$ and,
moreover, $\langle a_{m+1},b_{m+1}\rangle \in C$, therefore,
condition \eqref{E:cond-A} for $\langle A, B\rangle$ is
satisfied,and $\langle C, B\rangle \in [\mu](\{ A \})$.

The case $\mu \in \nu_x(\exists, \top)$ is similar.

\emph{Case 1.2}. Assume that $\mu \in \nu_x(\forall, f)$, where
$f$ is non-constant monotone Boolean function. Then we have $\{ A
\} \cup [\mu^0](\{ A \}) = \{ A \}$, and we need to show that for
every $B$, satisfying the lemma hypothesis, we have $\langle A,
B\rangle \in [\mu](\{ A \})$, that is to say, that $\langle A,
B\rangle$ satisfies condition \eqref{E:cond-A}. From our
assumption that

$$
B \in Rel(\{ \mu(\psi_1,\ldots, \psi_n)\mid \psi_1,\ldots, \psi_n
\in \mathcal{L}^\Theta_x(\mathbb{M}) \},M_1, M_2)
$$
we infer that:

\begin{equation}\label{E:refl-1}
B \in Rel(\{ \mu(\psi,\ldots, \psi)\mid \psi \in
\mathcal{L}^\Theta_x(\mathbb{M}) \},M_1, M_2)
\end{equation}

Since $\mu$ is $\forall$-guarded, we can assume that $\mu$ has a
form
$$
\forall x_2\ldots x_{m + 1}(\pi^{m}_1Sx \to \xi(P_1(x_{m +
1}),\ldots, P_n(x_{m + 1}))),
$$
in the assumptions of \eqref{E:Aguard}, where $\xi$ defines $f$
for $P_1(x_{m + 1}),\ldots, P_n(x_{m + 1})$. By Lemma
\ref{L:Boolean}.2 we get, for any $\psi \in
\mathcal{L}^\Theta_x(\mathbb{M})$, that:
$$
\forall x_2\ldots x_{m + 1}(\pi^{m}_1Sx \to \xi(\psi(x_{m +
1}),\ldots, \psi(x_{m + 1})))
$$
is logically equivalent to
$$
\forall x_2\ldots x_{m + 1}(\pi^{m}_1Sx \to \psi(x_{m + 1})).
$$
Therefore, from \eqref{E:refl-1} we can infer that

\begin{equation}\label{E:refl-2}
B \in Rel(\{ \mu'(\psi)\mid \psi \in
\mathcal{L}^\Theta_x(\mathbb{M}) \},M_1, M_2),
\end{equation}
where $\mu'$ is the following $x$-g.c.:

$$
\forall x_2\ldots x_{m + 1}(\pi^{m}_1Sx \to P_1(x_{m + 1})).
$$

We proceed now to verification of condition \eqref{E:cond-A}.
Assume that $\{ r,t \} = \{ 1,2 \}$ and that we have $a_1 \in
U_r$, $\bar{b}_{m+1} \in U_t$, and $a_1\mathrel{B}b_1$. Moreover,
assume that

$$
   \pi^m_1S^tb.
$$
Then consider the set
$Fa(\mathcal{L}^\Theta_{x_{m+1}}(\mathbb{M}), M_t, b_{m+1})$, that
is to say, $Fa^t_{x_{m+1}}(b_{m+1})$. Since $\bot \in
\mathcal{L}^\Theta_{x_{m+1}}(\mathbb{M})$, this set is non-empty,
and since $\vee \in \mathbb{M}$, then for every finite $\Delta
\subseteq Fa^t_{x_{m+1}}(b_{m+1})$, we have $\vee\Delta \in
Fa^t_{x_{m+1}}(b_{m+1})$. But then we have $b_1 \not\models_t
\mu'(\vee\Delta)$ and hence, by $a_1\mathrel{B}b_1$ and
\eqref{E:refl-2}, that $a_1 \not\models_r \mu'(\vee\Delta)$. But
the latter means that every finite subset of the set
$$
\{ S_1(a_1, x_2), \pi^m_2Sx \} \cup \neg Fa^t_{x_{m+1}}(b_{m+1})
$$
is satisfiable at  $M_r/a_1$. Therefore, by compactness of
first-order logic, this set is consistent with $Th(M_r/a_1)$ and,
by $\omega$-saturation of both $M_1$ and $M_2$, it must be
satisfied in $M_r/a_1$ by some $a_2,\ldots, a_{m+1}  \in U_r$. So
for any such $a_2,\ldots, a_{m+1}$ we will have both $\pi^m_1S^ra$
and, moreover
$$
a_{m+ 1} \models_r \neg Fa^t_{x_{m+1}}(b_{m+1}).
$$

Thus, by independence of truth at a sequence of elements from the
choice of free variables in a formula, we will also have
$$
a_{m+1} \models_r \neg Fa^t_{x}(b_{m+1}).
$$
This means that
$$
\{ \langle a_{m+1},b_{m+1}\rangle \} \in
Rel(\mathcal{L}^\Theta_x(\mathbb{M}),M_1, M_2),
$$
therefore, by definition of $A$ we get that
$a_{m+1}\mathrel{A}b_{m+1}$, and thus condition \eqref{E:cond-A}
for $\langle A, B\rangle$ is satisfied. Whence we conclude that
$\langle A, B\rangle \in [\mu](\{ A \})$.

\emph{Case 1.3}. Assume that $\mu \in \nu_x(\exists, f)$, where
$f$ is non-constant monotone Boolean function. Then we have $\{ A
\} \cup [\mu^0](\{ A \}) = \{ A \}$, and we need to show that for
every $B$, satisfying the lemma hypothesis, we have $\langle A,
B\rangle \in [\mu](\{ A \})$, that is to say, that $\langle A,
B\rangle$ satisfies condition \eqref{E:cond-E}. Arguing as in the
previous case, we infer \eqref{E:refl-1}.

Since $\mu$ is $\exists$-guarded, we can assume that $\mu$ has a
form
$$
\exists x_2\ldots x_{m + 1}(\pi^{m}_1Sx \wedge \xi(P_1(x_{m +
1}),\ldots, P_n(x_{m + 1}))),
$$
in the assumptions of \eqref{E:Eguard}, where $\xi$ defines $f$
for $P_1(x_{m + 1}),\ldots, P_n(x_{m + 1})$. By Lemma
\ref{L:Boolean}.2 we get, for any $\psi \in
\mathcal{L}^\Theta_x(\mathbb{M})$, that:
$$
\exists x_2\ldots x_{m + 1}(\pi^{m}_1Sx \wedge \xi(\psi(x_{m +
1}),\ldots, \psi(x_{m + 1})))
$$
is logically equivalent to
$$
\exists x_2\ldots x_{m + 1}(\pi^{m}_1Sx \wedge \psi(x_{m + 1})).
$$
Therefore, from \eqref{E:refl-1} we can infer that

\begin{equation}\label{E:refl-3}
B \in Rel(\{ \mu'(\psi)\mid \psi \in
\mathcal{L}^\Theta_x(\mathbb{M}) \},M_1, M_2),
\end{equation}
where $\mu'$ is the following $x$-g.c.:

$$
\exists x_2\ldots x_{m + 1}(\pi^{m}_1Sx \wedge P_1(x_{m + 1})).
$$

We proceed now to verification of condition \eqref{E:cond-E}.
Assume that $\{ r,t \} = \{ 1,2 \}$ and that we have
$\bar{a}_{m+1} \in U_r$, $b_1 \in U_t$, and $a_1\mathrel{B}b_1$.
Moreover, assume that $\pi^m_1S^ra$. Then consider the set
$Tr(\mathcal{L}^\Theta_{x_{m+1}}(\mathbb{M}), M_r, a_{m+1})$, that
is to say, $Tr^r_{x_{m+1}}(a_{m+1})$. Since $\top \in
\mathcal{L}^\Theta_{x_{m+1}}(\mathbb{M})$, this set is non-empty,
and since $\wedge \in \mathbb{M}$, then for every finite $\Gamma
\subseteq Tr^r_{x_{m+1}}(a_{m+1})$, we have $\wedge\Gamma \in
Tr^r_{x_{m+1}}(a_{m+1})$. But then we have $a_1 \models_r
\mu'(\wedge\Gamma)$ and hence, by $a_1\mathrel{B}b_1$ and
\eqref{E:refl-3}, that $b_1 \models_t \mu'(\wedge\Gamma)$. But the
latter means that every finite subset of the set
$$
\{ S_1(b_1, x_2), \pi^m_2Sx \} \cup Tr^r_{x_{m+1}}(a_{m+1})
$$
is satisfiable at  $M_t/b_1$. Therefore, by compactness of
first-order logic, this set is consistent with $Th(M_t/b_1)$ and,
by $\omega$-saturation of both $M_1$ and $M_2$, it must be
satisfied in $M_t/b_1$ by some $b_2,\ldots, b_{m+1}  \in U_t$. So
for any such $b_2,\ldots, b_{m+1}$ we will have both $\pi^m_1S^tb$
and, moreover
$$
b_{m+ 1} \models_t Tr^r_{x_{m+1}}(a_{m+1}).
$$

Thus, by independence of truth at a sequence of elements from the
choice of free variables in a formula, we will also have
$$
b_{m+1} \models_t Tr^r_{x}(a_{m+1})
$$
This means that
$$
\{ \langle a_{m+1},b_{m+1}\rangle \} \in
Rel(\mathcal{L}^\Theta_x(\mathbb{M}),M_1, M_2),
$$
therefore, by definition of $A$ we get that
$a_{m+1}\mathrel{A}b_{m+1}$, and thus condition \eqref{E:cond-E}
for $\langle A, B\rangle$ is satisfied. Whence we conclude that
$\langle A, B\rangle \in [\mu](\{ A \})$.

\emph{Case 1.4}. Assume that $\mu \in \nu_x(\forall, f)$, where
$f$ is non-constant anti-monotone Boolean function. Then we have
$\{ A \} \cup [\mu^0](\{ A \}) = \{ A, A^{-1} \}$. We will show
that for every $B$, satisfying the lemma hypothesis, we have
$\langle A^{-1}, B\rangle \in [\mu](\{ A \})$, that is to say,
that $\langle A^{-1}, B\rangle$ satisfies condition
\eqref{E:cond-A}. Arguing as in the previous case, we infer
\eqref{E:refl-1}.

Since $\mu$ is $\forall$-guarded, we can assume that $\mu$ has a
form
$$
\forall x_2\ldots x_{m + 1}(\pi^{m}_1Sx \to \xi(P_1(x_{m +
1}),\ldots, P_n(x_{m + 1}))),
$$
in the assumptions of \eqref{E:Aguard}, where $\xi$ defines $f$
for $P_1(x_{m + 1}),\ldots, P_n(x_{m + 1})$. By Lemma
\ref{L:Boolean}.4 we get, for any $\psi \in
\mathcal{L}^\Theta_x(\mathbb{M})$, that:
$$
\forall x_2\ldots x_{m + 1}(\pi^{m}_1Sx \to \xi(\psi(x_{m +
1}),\ldots, \psi(x_{m + 1})))
$$
is logically equivalent to
$$
\forall x_2\ldots x_{m + 1}(\pi^{m}_1Sx \to \neg\psi(x_{m + 1})).
$$
Therefore, from \eqref{E:refl-1} we can infer that

\begin{equation}\label{E:refl-4}
B \in Rel(\{ \mu'(\psi)\mid \psi \in
\mathcal{L}^\Theta_x(\mathbb{M}) \},M_1, M_2),
\end{equation}
where $\mu'$ is the following $x$-g.c.:

$$
\forall x_2\ldots x_{m + 1}(\pi^{m}_1Sx \to \neg P_1(x_{m + 1})).
$$

We proceed now to verification of condition \eqref{E:cond-A}.
Assume that $\{ r,t \} = \{ 1,2 \}$ and that we have $a_1 \in
U_r$, $\bar{b}_{m+1} \in U_t$, and $a_1\mathrel{B}b_1$. Moreover,
assume that $\pi^m_1S^tb$. Then consider the set
$Tr^t_{x_{m+1}}(b_{m+1})$. This set is non-empty, and since
$\wedge \in \mathbb{M}$, then for every finite $\Gamma \subseteq
Tr^t_{x_{m+1}}(b_{m+1})$, we have $\wedge\Gamma \in
Tr^t_{x_{m+1}}(b_{m+1})$. But then we have $b_1 \not\models_t
\mu'(\wedge\Gamma)$ and hence, by $a_1\mathrel{B}b_1$ and
\eqref{E:refl-4}, that $a_1 \not\models_r \mu'(\wedge\Gamma)$. But
the latter means that every finite subset of the set
$$
\{ S_1(a_1, x_2), \pi^m_2Sx \} \cup Tr^t_{x_{m+1}}(b_{m+1})
$$
is satisfiable at  $M_r/a_1$. Therefore, by compactness of
first-order logic, this set is consistent with $Th(M_r/a_1)$ and,
by $\omega$-saturation of both $M_1$ and $M_2$, it must be
satisfied in $M_r/a_1$ by some $a_2,\ldots, a_{m+1}  \in U_r$. So
for any such $a_2,\ldots, a_{m+1}$ we will have both $\pi^m_1S^ra$
and, moreover
$$
a_{m+ 1} \models_r Tr^t_{x_{m+1}}(b_{m+1}).
$$
Thus, by independence of truth at a sequence of elements from the
choice of free variables in a formula, we will also have
$$
a_{m+1} \models_r Tr^t_{x}(b_{m+1}).
$$
This means that
$$
\{ \langle b_{m+1},a_{m+1}\rangle \} \in
Rel(\mathcal{L}^\Theta_x(\mathbb{M}),M_1, M_2),
$$
therefore, by definition of $A$ we get that
$b_{m+1}\mathrel{A}a_{m+1}$. Hence we get
$a_{m+1}\mathrel{A^{-1}}b_{m+1}$ and thus condition
\eqref{E:cond-A} for $\langle A^{-1}, B\rangle$ is satisfied.
Whence we conclude that

\noindent$\langle A^{-1}, B\rangle \in [\mu](\{ A \})$.

\emph{Case 1.5}. Assume that $\mu \in \nu_x(\exists, f)$, where
$f$ is non-constant anti-monotone Boolean function. Then we have
$\{ A \} \cup [\mu^0](\{ A \}) = \{ A, A^{-1} \}$. We will show
that for every $B$, satisfying the lemma hypothesis, we have
$\langle A^{-1}, B\rangle \in [\mu](\{ A \})$, that is to say,
that $\langle A^{-1}, B\rangle$ satisfies condition
\eqref{E:cond-E}. Arguing as in the previous case, we infer
\eqref{E:refl-1}.

Since $\mu$ is $\exists$-guarded, we can assume that $\mu$ has a
form
$$
\exists x_2\ldots x_{m + 1}(\pi^{m}_1Sx \wedge \xi(P_1(x_{m +
1}),\ldots, P_n(x_{m + 1}))),
$$
in the assumptions of \eqref{E:Eguard}, where $\xi$ defines $f$
for $P_1(x_{m + 1}),\ldots, P_n(x_{m + 1})$. By Lemma
\ref{L:Boolean}.4 we get, for any $\psi \in
\mathcal{L}^\Theta_x(\mathbb{M})$, that:
$$
\exists x_2\ldots x_{m + 1}(\pi^{m}_1Sx \wedge \xi(\psi(x_{m +
1}),\ldots, \psi(x_{m + 1})))
$$
is logically equivalent to
$$
\exists x_2\ldots x_{m + 1}(\pi^{m}_1Sx \wedge \neg\psi(x_{m +
1})).
$$
Therefore, from \eqref{E:refl-1} we can infer that

\begin{equation}\label{E:refl-5}
B \in Rel(\{ \mu'(\psi)\mid \psi \in
\mathcal{L}^\Theta_x(\mathbb{M}) \},M_1, M_2),
\end{equation}
where $\mu'$ is the following $x$-g.c.:

$$
\exists x_2\ldots x_{m + 1}(\pi^{m}_1Sx \wedge \neg P_1(x_{m +
1})).
$$

We proceed now to verification of condition \eqref{E:cond-E}.
Assume that $\{ r,t \} = \{ 1,2 \}$ and that we have
$\bar{a}_{m+1} \in U_r$, $b_1 \in U_t$, and $a_1\mathrel{B}b_1$.
Moreover, assume that $\pi^m_1S^ra$. Then consider the set
$Fa^r_{x}(a_{m+1})$. This set is non-empty, and since $\vee \in
\mathbb{M}$, then for every finite $\Delta \subseteq
Fa^r_{x_{m+1}}(a_{m+1})$, we have $ \vee\Delta \in
Fa^r_{x_{m+1}}(a_{m+1})$. But then we have $a_1 \models_r
\mu'(\vee\Delta)$ and hence, by $a_1\mathrel{B}b_1$ and
\eqref{E:refl-5}, that $b_1 \models_t \mu'(\vee\Delta)$. But the
latter means that every finite subset of the set
$$
\{ S_1(b_1, x_2), \pi^m_2Sx \} \cup \neg Fa^r_{x_{m+1}}(a_{m+1})
$$
is satisfiable at  $M_t/b_1$. Therefore, by compactness of
first-order logic, this set is consistent with $Th(M_t/b_1)$ and,
by $\omega$-saturation of both $M_1$ and $M_2$, it must be
satisfied in $M_t/b_1$ by some $b_2,\ldots, b_{m+1}  \in U_t$. So
for any such $b_2,\ldots, b_{m+1}$ we will have both $\pi^m_1S^tb$
and, moreover
$$
b_{m+ 1} \models_t  \neg Fa^r_{x_{m+1}}(a_{m+1}).
$$

Thus, by independence of truth at a sequence of elements from the
choice of free variables in a formula, we will also have
$$
b_{m+1} \models_t  \neg Fa^r_{x}(a_{m+1}).
$$
This means that
$$
\{ \langle b_{m+1},a_{m+1}\rangle \} \in
Rel(\mathcal{L}^\Theta_x(\mathbb{M}),M_1, M_2),
$$
therefore, by definition of $A$ we get that
$b_{m+1}\mathrel{A}a_{m+1}$. Hence we have
$a_{m+1}\mathrel{A^{-1}}b_{m+1}$, and thus condition
\eqref{E:cond-E} for $\langle A^{-1}, B\rangle$ is satisfied.
Whence we conclude that $\langle A^{-1}, B\rangle \in [\mu](\{ A
\})$.

\emph{Case 1.6}. Let $\mu \in \nu_x(\forall, f)$, where $f$ is a
rest Boolean $TFT$-function. Then we have $\{ A \} \cup [\mu^0](\{
A \}) = \{ A, A \cap A^{-1} \}$. We will show that for every $B$,
satisfying the lemma hypothesis, we have $\langle A \cap A^{-1},
B\rangle \in [\mu](\{ A \})$, that is to say, that $\langle A \cap
A^{-1}, B\rangle$ satisfies condition \eqref{E:cond-A}.

Since $\mu$ is $\forall$-guarded, we can assume that $\mu$ has the
form
$$
\forall x_2\ldots x_{m + 1}(\pi^{m}_1Sx \to \xi(P_1(x_{m +
1}),\ldots, P_n(x_{m + 1}))),
$$
in the assumptions of \eqref{E:Aguard}, where $\xi$ defines $f$
for $P_1(x_{m + 1}),\ldots, P_n(x_{m + 1})$. By Lemma
\ref{L:Boolean}.5 we get that for arbitrary $\psi_1,\psi_2 \in
\mathcal{L}^\Theta_x(\mathbb{M})$ there exist
$\tau_1,\ldots,\tau_n \in \{ \psi_1, \psi_2, \psi_1 \wedge \psi_2,
\top, \bot \}$ such that the formula

$$
\forall x_2\ldots x_{m + 1}(\pi^{m}_1Sx \to \xi(\tau_1(x_{m +
1}),\ldots, \tau_n(x_{m + 1})))
$$
is logically equivalent to
$$
\forall x_2\ldots x_{m + 1}(\pi^{m}_1Sx \to (\neg\psi_1(x_{m +
1})\vee \psi_2(x_{m + 1}))).
$$
Therefore, by our assumptions that $\mathbb{M} \supseteq \{
\wedge, \vee, \top, \bot \}$ and that
$$
B \in Rel(\{ \mu(\psi_1,\ldots, \psi_n)\mid \psi_1,\ldots, \psi_n
\in \mathcal{L}^\Theta_x(\mathbb{M}) \},M_1, M_2)
$$
we infer that:
\begin{equation}\label{E:refl-6}
B \in Rel(\{ \mu'(\psi_1,\psi_2)\mid \psi_1, \psi_2 \in
\mathcal{L}^\Theta_x(\mathbb{M}) \},M_1, M_2),
\end{equation}

where $\mu'$ is the following $x$-g.c.:

$$
\forall x_2\ldots x_{m + 1}(\pi^{m}_1Sx \to (\neg P_1(x_{m + 1})
\vee P_2(x_{m + 1}))),
$$

We proceed now to verification of condition \eqref{E:cond-A}.
Assume that $\{ r,t \} = \{ 1,2 \}$ and that we have $a_1 \in
U_r$, $\bar{b}_{m+1} \in U_t$, and $a_1\mathrel{B}b_1$. Moreover,
assume that $\pi^m_1S^tb$. Then consider the sets
$Tr^t_{x_{m+1}}(b_{m+1})$ and $Fa^t_{x_{m+1}}(b_{m+1})$. These
sets are both non-empty, and since $\wedge, \vee \in \mathbb{M}$,
then for every finite $\Gamma \subseteq Tr^t_{x_{m+1}}(b_{m+1})$
and every finite $\Delta \subseteq Fa^t_{x_{m+1}}(b_{m+1})$, we
have
$$
\wedge\Gamma \in Tr^t_{x_{m+1}}(b_{m+1}), \vee \Delta\in
Fa^t_{x_{m+1}}(b_{m+1}).
$$
But then we have $b_1 \not\models_t \mu'(\wedge\Gamma, \vee
\Delta)$ and hence, by $a_1\mathrel{B}b_1$ and \eqref{E:refl-6},
that $a_1 \not\models_r \mu'(\wedge\Gamma, \vee \Delta)$. But the
latter means that every finite subset of the set
$$
\{ S_1(a_1, x_2), \pi^m_2Sx \} \cup Tr^t_{x_{m+1}}(b_{m+1}) \cup
\neg Fa^t_{x_{m+1}}(b_{m+1})
$$
is satisfiable at  $M_r/a_1$. Therefore, by compactness of
first-order logic, this set is consistent with $Th(M_r/a_1)$ and,
by $\omega$-saturation of both $M_1$ and $M_2$, it must be
satisfied in $M_r/a_1$ by some $a_2,\ldots, a_{m+1}  \in U_r$. So
for any such $a_2,\ldots, a_{m+1}$ we will have both $\pi^m_1S^ra$
and, moreover
$$
a_{m+ 1} \models_r Tr^t_{x_{m+1}}(b_{m+1}) \cup \neg
Fa^t_{x_{m+1}}(b_{m+1}).
$$

Thus, by independence of truth at a sequence of elements from the
choice of free variables in a formula, we will also have
$$
a_{m+1} \models_r Tr^t_{x}(b_{m+1}) \cup \neg Fa^t_x(b_{m+1}).
$$
This means that
$$
\{ \langle b_{m+1},a_{m+1}\rangle, \langle a_{m+1},b_{m+1}\rangle
\} \in Rel(\mathcal{L}^\Theta_x(\mathbb{M}),M_1, M_2),
$$
therefore, by definition of $A$ we get that
$a_{m+1}\mathrel{A}b_{m+1}$ and $b_{m+1}\mathrel{A}a_{m+1}$. Hence
we get $a_{m+1}\mathrel{A \cap A^{-1}}b_{m+1}$ and thus condition
\eqref{E:cond-A} for $\langle A \cap A^{-1}, B\rangle$ is
satisfied. Whence we conclude that $\langle A \cap A^{-1},
B\rangle \in [\mu](\{ A \})$.

\emph{Case 1.7}. Let $\mu \in \nu_x(\exists, f)$, where $f$ is a
rest Boolean $FTF$-function. Then we have $\{ A \} \cup [\mu^0](\{
A \}) = \{ A, A \cap A^{-1} \}$. We will show that for every $B$
satisfying the lemma hypothesis, we have $\langle A \cap A^{-1},
B\rangle \in [\mu](\{ A \})$, that is to say, that $\langle A \cap
A^{-1}, B\rangle$ satisfies condition \eqref{E:cond-E}.

Since $\mu$ is $\exists$-guarded, we can assume that $\mu$ has the
form
$$
\exists x_2\ldots x_{m + 1}(\pi^{m}_1Sx \wedge \xi(P_1(x_{m +
1}),\ldots, P_n(x_{m + 1}))),
$$
in the assumptions of \eqref{E:Eguard}, where $\xi$ defines $f$
for $P_1(x_{m + 1}),\ldots, P_n(x_{m + 1})$. By Lemma
\ref{L:Boolean}.6 we get that for arbitrary $\psi_1,\psi_2 \in
\mathcal{L}^\Theta_x(\mathbb{M})$ there exist
$\tau_1,\ldots,\tau_n \in \{ \psi_1, \psi_2, \psi_1 \wedge \psi_2,
\top, \bot \}$ such that the formula

$$
\exists x_2\ldots x_{m + 1}(\pi^{m}_1Sx \wedge \xi(\tau_1(x_{m +
1}),\ldots, \tau_n(x_{m + 1})))
$$
is logically equivalent to
$$
\exists x_2\ldots x_{m + 1}(\pi^{m}_1Sx \wedge (\psi_1(x_{m +
1})\wedge \neg\psi_2(x_{m + 1}))).
$$
Therefore, by our assumptions that $\mathbb{M} \supseteq \{
\wedge, \vee, \top, \bot \}$ and that
$$
B \in Rel(\{ \mu(\psi_1,\ldots, \psi_n)\mid \psi_1,\ldots, \psi_n
\in \mathcal{L}^\Theta_x(\mathbb{M}) \},M_1, M_2)
$$
we infer that:
\begin{equation}\label{E:refl-7}
B \in Rel(\{ \mu'(\psi_1,\psi_2)\mid \psi_1, \psi_2 \in
\mathcal{L}^\Theta_x(\mathbb{M}) \},M_1, M_2),
\end{equation}

where $\mu'$ is the following $x$-g.c.:

$$
\exists x_2\ldots x_{m + 1}(\pi^{m}_1Sx \wedge (P_1(x_{m + 1})
\wedge \neg P_2(x_{m + 1}))),
$$

We proceed now to verification of condition \eqref{E:cond-E}.
Assume that $\{ r,t \} = \{ 1,2 \}$ and that we have
$\bar{a}_{m+1} \in U_r$, $b_1 \in U_t$, and $a_1\mathrel{B}b_1$.
Moreover, assume that $\pi^m_1S^ra$. Then consider the sets
$Tr^r_{x_{m+1}}(a_{m+1})$ and $Fa^r_{x_{m+1}}(a_{m+1})$. These
sets are non-empty, and since $\wedge,\vee \in \mathbb{M}$, then
for every finite $\Gamma \subseteq Tr^r_{x_{m+1}}(a_{m+1})$ and
every finite $\Delta \subseteq Fa^r_{x_{m+1}}(a_{m+1})$, we have
$$
\wedge\Gamma \in Tr^r_{x_{m+1}}(a_{m+1}), \vee\Delta \in
Fa^r_{x_{m+1}}(a_{m+1}).
$$
But then we have $a_1 \models_r \mu'(\wedge\Gamma,\vee\Delta)$ and
hence, by $a_1\mathrel{B}b_1$ and \eqref{E:refl-7}, that $b_1
\models_t \mu'(\wedge\Gamma,\vee\Delta)$. But the latter means
that every finite subset of the set
$$
\{ S_1(b_1, x_2), \pi^m_2Sx \} \cup Tr^r_{x_{m+1}}(a_{m+1})\cup
\neg Fa^r_{x_{m+1}}(a_{m+1})
$$
is satisfiable at  $M_t/b_1$. Therefore, by compactness of
first-order logic, this set is consistent with $Th(M_t/b_1)$ and,
by $\omega$-saturation of both $M_1$ and $M_2$, it must be
satisfied in $M_t/b_1$ by some $b_2,\ldots, b_{m+1}  \in U_t$. So
for any such $b_2,\ldots, b_{m+1}$ we will have both $\pi^m_1S^tb$
and, moreover
$$
b_{m+ 1} \models_t Tr^r_{x_{m+1}}(a_{m+1})\cup \neg
Fa^r_{x_{m+1}}(a_{m+1}).
$$

Thus, by independence of truth at a sequence of elements from the
choice of free variables in a formula, we will also have
$$
b_{m+1} \models_t Tr^r_{x}(a_{m+1})\cup \neg Fa^r_{x}(a_{m+1}).
$$
This means that
$$
\{ \langle b_{m+1},a_{m+1}\rangle, \langle a_{m+1},b_{m+1}\rangle
\} \in Rel(\mathcal{L}^\Theta_x(\mathbb{M}),M_1, M_2),
$$
therefore, by definition of $A$ we get that
$a_{m+1}\mathrel{A}b_{m+1}$ and $b_{m+1}\mathrel{A}a_{m+1}$. Hence
we get $a_{m+1}\mathrel{A \cap A^{-1}}b_{m+1}$ and thus condition
\eqref{E:cond-E} for $\langle A \cap A^{-1}, B\rangle$ is
satisfied. Whence we conclude that $\langle A \cap A^{-1},
B\rangle \in [\mu](\{ A \})$.

\emph{Case 1.8}. Let $\mu \in \nu_x(\forall, f)$, where $f$ is a
rest Boolean non-$TFT$ function. Then $\mu$ is a special guarded
connective and we have $\{ A \} \cup [\mu^0](\{ A \}) = \{ A, A
\cap A^{-1} \}$. We will show that for every $B$, satisfying the
lemma hypothesis, we have $\langle A, B\rangle \in [\mu](\{ A
\})$, that is to say, that $\langle A, B\rangle$ satisfies
condition \eqref{E:cond-Aspec}.

Since $\mu$ is $\forall$-guarded, we can assume that $\mu$ has the
form
$$
\forall x_2\ldots x_{m + 1}(\pi^{m}_1Sx \to \xi(P_1(x_{m +
1}),\ldots, P_n(x_{m + 1}))),
$$
in the assumptions of \eqref{E:Aguard}, where $\xi$ defines $f$
for $P_1(x_{m + 1}),\ldots, P_n(x_{m + 1})$. By Lemma
\ref{L:Boolean}.7 we get that for arbitrary $\psi_1,\psi_2 \in
\mathcal{L}^\Theta_x(\mathbb{M})$ there exist

$$
\tau_1,\ldots,\tau_n, \theta_1,\ldots, \theta_n \in \{ \psi_1,
\top, \bot \}
$$

such that the formula
$$
\forall x_2\ldots x_{m + 1}(\pi^{m}_1Sx \to \xi(\tau_1(x_{m +
1}),\ldots, \tau_n(x_{m + 1})))
$$
is logically equivalent to
$$
\forall x_2\ldots x_{m + 1}(\pi^{m}_1Sx \to \psi_1(x_{m + 1})),
$$
whereas the formula
$$
\forall x_2\ldots x_{m + 1}(\pi^{m}_1Sx \to \xi(\theta_1(x_{m +
1}),\ldots, \theta_n(x_{m + 1})))
$$
is logically equivalent to
$$
\forall x_2\ldots x_{m + 1}(\pi^{m}_1Sx \to \neg\psi_1(x_{m +
1})).
$$

Therefore, by our assumptions that $\mathbb{M} \supseteq \{
\wedge, \vee, \top, \bot \}$ and that
$$
B \in Rel(\{ \mu(\psi_1,\ldots, \psi_n)\mid \psi_1,\ldots, \psi_n
\in \mathcal{L}^\Theta_x(\mathbb{M}) \},M_1, M_2)
$$
we infer that:
\begin{equation}\label{E:refl-8}
B \in Rel(\{ \mu'(\psi)\mid \psi \in
\mathcal{L}^\Theta_x(\mathbb{M}) \},M_1, M_2) \cap Rel(\{
\mu''(\psi)\mid \psi \in \mathcal{L}^\Theta_x(\mathbb{M}) \},M_1,
M_2),
\end{equation}

where $\mu'$ and $\mu''$ are defined as follows:

$$
\mu' = \forall x_2\ldots x_{m + 1}(\pi^{m}_1Sx \to P_1(x_{m +
1})),
$$
$$
\mu'' = \forall x_2\ldots x_{m + 1}(\pi^{m}_1Sx \to \neg P_1(x_{m
+ 1})).
$$

We proceed now to verification of condition \eqref{E:cond-Aspec}.
Assume that $\{ r,t \} = \{ 1,2 \}$ and that we have $a_1 \in
U_r$, $\bar{b}_{m+1} \in U_t$, and $a_1\mathrel{B}b_1$. Moreover,
assume that $\pi^m_1S^tb$. Then, first, consider the non-empty set
$Fa^t_{x_{m+1}}(b_{m+1})$. Since $\vee \in \mathbb{M}$, then for
every finite $\Delta \subseteq Fa^t_{x_{m+1}}(b_{m+1})$, we have
$\vee \Delta \in Fa^t_{x_{m+1}}(b_{m+1})$. But then we have $b_1
\not\models_t \mu'(\vee\Delta)$ and hence, by $a_1\mathrel{B}b_1$
and \eqref{E:refl-8}, that $a_1 \not\models_r \mu'(\vee\Delta)$.
But the latter means that every finite subset of the set
$$
\{ S_1(a_1, x_2), \pi^m_2Sx \} \cup \neg Fa^t_{x_{m+1}}(b_{m+1})
$$
is satisfiable at  $M_r/a_1$. Therefore, by compactness of
first-order logic, this set is consistent with $Th(M_r/a_1)$ and,
by $\omega$-saturation of both $M_1$ and $M_2$, it must be
satisfied in $M_r/a_1$ by some $a_2,\ldots, a_{m+1}  \in U_r$. So
for any such $a_2,\ldots, a_{m+1}$ we will have both $\pi^m_1S^ra$
and, moreover
$$
a_{m+ 1} \models_r \neg Fa^t_{x_{m+1}}(b_{m+1}).
$$

Thus, by independence of truth at a sequence of elements from the
choice of free variables in a formula, we will also have
$$
a_{m+1} \models_r \neg Fa^t_{x}(b_{m+1}).
$$
This means that $\{ \langle a_{m+1},b_{m+1}\rangle \} \in
Rel(\mathcal{L}^\Theta_x(\mathbb{M}),M_1, M_2)$, so that, by
definition of $A$, we get that $a_{m+1}\mathrel{A}b_{m+1}$.

Now, second, consider the non-empty set $Tr^t_{x_{m+1}}(b_{m+1})$.
Since $\wedge \in \mathbb{M}$, then for every finite $\Gamma
\subseteq Tr^t_{x_{m+1}}(b_{m+1})$ we have $\wedge \Gamma \in
Tr^t_{x_{m+1}}(b_{m+1})$. But then we have $b_1 \not\models_t
\mu''(\wedge\Gamma)$ and hence, by $a_1\mathrel{B}b_1$ and
\eqref{E:refl-8}, that $a_1 \not\models_r \mu''(\wedge\Gamma)$.
But the latter means that every finite subset of the set
$$
\{ S_1(a_1, x_2), \pi^m_2Sx \} \cup Tr^t_{x_{m+1}}(b_{m+1})
$$
is satisfiable at  $M_r/a_1$. Therefore, by compactness of
first-order logic, this set is consistent with $Th(M_r/a_1)$ and,
by $\omega$-saturation of both $M_1$ and $M_2$, it must be
satisfied in $M_r/a_1$ by some $c_2,\ldots, c_{m+1}  \in U_r$. So
for any such $c_2,\ldots, c_{m+1}$ we will have both $S^r_1(a_1,
c_2) \wedge \pi^m_2S^rc$ and, moreover
$$
c_{m+ 1} \models_r Tr^t_{x_{m+1}}(b_{m+1}).
$$

Thus, by independence of truth at a sequence of elements from the
choice of free variables in a formula, we will also have
$$
c_{m+1} \models_r Tr^t_{x}(b_{m+1}).
$$
This means that
$$
\{ \langle b_{m+1},c_{m+1}\rangle \} \in
Rel(\mathcal{L}^\Theta_x(\mathbb{M}),M_1, M_2),
$$
therefore, by definition of $A$ we get that
$b_{m+1}\mathrel{A}c_{m+1}$ and hence
$c_{m+1}\mathrel{A^{-1}}b_{m+1}$.Thus condition
\eqref{E:cond-Aspec} for $\langle A, B\rangle$ is satisfied and we
conclude that $\langle A, B\rangle \in [\mu](\{ A \})$.

\emph{Case 1.9}. Let $\mu \in \nu_x(\exists, f)$, where $f$ is a
rest Boolean non-$FTF$ function. Then $\mu$ is a special guarded
connective and we have $\{ A \} \cup [\mu^0](\{ A \}) = \{ A, A
\cap A^{-1} \}$. We will show that for every $B$, satisfying the
lemma hypothesis, we have $\langle A, B\rangle \in [\mu](\{ A
\})$, that is to say, that $\langle A, B\rangle$ satisfies
condition \eqref{E:cond-Espec}.

Since $\mu$ is $\forall$-guarded, we can assume that $\mu$ has the
form
$$
\exists x_2\ldots x_{m + 1}(\pi^{m}_1Sx \wedge \xi(P_1(x_{m +
1}),\ldots, P_n(x_{m + 1}))),
$$
in the assumptions of \eqref{E:Eguard}, where $\xi$ defines $f$
for $P_1(x_{m + 1}),\ldots, P_n(x_{m + 1})$. By Lemma
\ref{L:Boolean}.7 we get that for arbitrary $\psi_1,\psi_2 \in
\mathcal{L}^\Theta_x(\mathbb{M})$ there exist

$$
\tau_1,\ldots,\tau_n, \theta_1,\ldots, \theta_n \in \{ \psi_1,
\top, \bot \}
$$

such that the formula
$$
\exists x_2\ldots x_{m + 1}(\pi^{m}_1Sx \wedge \xi(\tau_1(x_{m +
1}),\ldots, \tau_n(x_{m + 1})))
$$
is logically equivalent to
$$
\exists x_2\ldots x_{m + 1}(\pi^{m}_1Sx \wedge \psi_1(x_{m + 1})),
$$
whereas the formula
$$
\exists x_2\ldots x_{m + 1}(\pi^{m}_1Sx \wedge \xi(\theta_1(x_{m +
1}),\ldots, \theta_n(x_{m + 1})))
$$
is logically equivalent to
$$
\exists x_2\ldots x_{m + 1}(\pi^{m}_1Sx \wedge \neg\psi_1(x_{m +
1})).
$$

Therefore, by our assumptions that $\mathbb{M} \supseteq \{
\wedge, \vee, \top, \bot \}$ and that
$$
B \in Rel(\{ \mu(\psi_1,\ldots, \psi_n)\mid \psi_1,\ldots, \psi_n
\in \mathcal{L}^\Theta_x(\mathbb{M}) \},M_1, M_2)
$$
we infer that:
\begin{equation}\label{E:refl-9}
B \in Rel(\{ \mu'(\psi)\mid \psi \in
\mathcal{L}^\Theta_x(\mathbb{M}) \},M_1, M_2) \cap Rel(\{
\mu''(\psi)\mid \psi \in \mathcal{L}^\Theta_x(\mathbb{M}) \},M_1,
M_2),
\end{equation}

where $\mu'$ and $\mu''$ are defined as follows:

$$
\mu' = \exists x_2\ldots x_{m + 1}(\pi^{m}_1Sx \wedge P_1(x_{m +
1})),
$$
$$
\mu'' = \exists x_2\ldots x_{m + 1}(\pi^{m}_1Sx \wedge \neg
P_1(x_{m + 1})).
$$

We proceed now to verification of condition \eqref{E:cond-Espec}.
Assume that $\{ r,t \} = \{ 1,2 \}$ and that we have
$\bar{a}_{m+1} \in U_r$, $b_1 \in U_t$, and $a_1\mathrel{B}b_1$.
Moreover, assume that $\pi^m_1S^ra$. Then, first, consider the
non-empty set $Tr^r_{x_{m+1}}(a_{m+1})$. Since $\wedge \in
\mathbb{M}$, then for every finite $\Gamma \subseteq
Tr^r_{x_{m+1}}(a_{m+1})$, we have $\wedge\Gamma \in
Tr^r_{x_{m+1}}(a_{m+1})$. But then we have $a_1 \models_r
\mu'(\wedge\Gamma)$ and hence, by $a_1\mathrel{B}b_1$ and
\eqref{E:refl-9}, that $b_1 \models_t \mu'(\wedge\Gamma)$. But the
latter means that every finite subset of the set
$$
\{ S_1(b_1, x_2), \pi^m_2Sx \} \cup Tr^r_{x_{m+1}}(a_{m+1})
$$
is satisfiable at  $M_t/b_1$. Therefore, by compactness of
first-order logic, this set is consistent with $Th(M_t/b_1)$ and,
by $\omega$-saturation of both $M_1$ and $M_2$, it must be
satisfied in $M_t/b_1$ by some $b_2,\ldots, b_{m+1}  \in U_t$. So
for any such $b_2,\ldots, b_{m+1}$ we will have both $\pi^m_1S^tb$
and, moreover
$$
b_{m+ 1} \models_t Tr^r_{x_{m+1}}(a_{m+1}).
$$

Thus, by independence of truth at a sequence of elements from the
choice of free variables in a formula, we will also have
$$
b_{m+1} \models_t Tr^r_{x}(a_{m+1}).
$$
This means that $\{ \langle a_{m+1},b_{m+1}\rangle \} \in
Rel(\mathcal{L}^\Theta_x(\mathbb{M}),M_1, M_2)$, therefore, by
definition of $A$ we get that $a_{m+1}\mathrel{A}b_{m+1}$.

Now, second, consider the non-empty set $Fa^r_{x_{m+1}}(a_{m+1})$.
Since $\vee \in \mathbb{M}$, then for every finite $\Delta
\subseteq Fa^r_{x_{m+1}}(a_{m+1})$ we have $\vee\Delta \in
Fa^r_{x_{m+1}}(a_{m+1})$. But then we have $a_1 \models_r
\mu''(\vee\Delta)$ and hence, by $a_1\mathrel{B}b_1$ and
\eqref{E:refl-9}, that $b_1 \models_t \mu''(\vee\Delta)$. But the
latter means that every finite subset of the set
$$
\{ S_1(b_1, x_2), \pi^m_2Sx \} \cup \neg Fa^r_{x_{m+1}}(a_{m+1})
$$
is satisfiable at  $M_t/b_1$. Therefore, by compactness of
first-order logic, this set is consistent with $Th(M_t/b_1)$ and,
by $\omega$-saturation of both $M_1$ and $M_2$, it must be
satisfied in $M_t/b_1$ by some $c_2,\ldots, c_{m+1}  \in U_t$. So
for any such $c_2,\ldots, c_{m+1}$ we will have both $S^t_1(b_1,
c_2) \wedge \pi^m_2S^tc$ and, moreover
$$
c_{m+ 1} \models_t \neg Fa^r_{x_{m+1}}(a_{m+1}).
$$

Thus, by independence of truth at a sequence of elements from the
choice of free variables in a formula, we will also have
$$
c_{m+1} \models_t \neg Fa^r_{x}(a_{m+1}).
$$
This means that $\{ \langle c_{m+1},a_{m+1}\rangle \} \in
Rel(\mathcal{L}^\Theta_x(\mathbb{M}),M_1, M_2)$, therefore, by
definition of $A$ we get that $c_{m+1}\mathrel{A}a_{m+1}$ and
hence $a_{m+1}\mathrel{A^{-1}}c_{m+1}$. Thus condition
\eqref{E:cond-Espec} for $\langle A, B\rangle$ is satisfied.
Whence we conclude that $\langle A, B\rangle \in [\mu](\{ A \})$.

\emph{Case 2}. Now, assume that $\delta(\mu) = 2$ and $\mu$ is a
regular guarded connective. Then we can assume that Lemma is
already proved for $\mu^-$. We have to distinguish between the
following cases:

\emph{Case 2.1}. $\mu \in \nu_x(\forall\exists, f)$, where $f$ is
a non-constant Boolean monotone function.

Then we can assume that $\mu$ has a form $\forall x_2\ldots x_{m +
1}(\pi^{m}_1Sx \to \mu^-)$ in the assumptions of \eqref{E:Aguard}
and that $\mu^- \in \nu_x(\exists, f)$.

Consider $A_2$ as defined in lemma. We have of course

\begin{equation}\label{E:A21}
A_2 \in Rel(\{ \mu^-(\psi_1,\ldots, \psi_n)\mid \psi_1,\ldots,
\psi_n \in \mathcal{L}^\Theta_x(\mathbb{M}) \},M_1, M_2).
\end{equation}

Therefore, by induction hypothesis, for some $C \in \{ A \} \cup
[\mu^0](\{ A \})$ the couple $\langle C,A_2\rangle$ is in
$[\mu^-](\{ A \})$.

Now assume that
$$
B \in Rel(\{ \mu(\psi_1,\ldots, \psi_n)\mid \psi_1,\ldots, \psi_n
\in \mathcal{L}^\Theta_x(\mathbb{M}) \},M_1, M_2).
$$
To show that $\langle C,A_2, B\rangle$ is in $[\mu](\{ A \})$, we
only need to verify condition \eqref{E:cond-A}.

So assume that $\{ r,t \} = \{ 1,2 \}$ and that we have $a_1 \in
U_r$, $\bar{b}_{m+1} \in U_t$, and $a_1\mathrel{B}b_1$. Moreover,
assume that $\pi^m_1S^tb$. Then consider the set
$$
\mathbb{F} = Fa(\{ \mu^-(\psi,\ldots,\psi)\mid \psi \in
\mathcal{L}^\Theta_{x_{m+1}}(\mathbb{M}) \}, M_t, b_{m+1}).
$$
This set is non-empty, since we have $\bot \in
\mathcal{L}^\Theta_{x_{m+1}}(\mathbb{M})$, and, further:
$$
\bot \Leftrightarrow \mu^-(\bot,\ldots,\bot) \in \mathbb{F}.
$$
Now, take an arbitrary finite subset

\begin{equation}\label{E:set1}
\{ \mu^-(\psi_1,\ldots, \psi_1),\ldots, \mu^-(\psi_u,\ldots,
\psi_u) \} \subseteq \mathbb{F}.
\end{equation}
Note that since we have $\psi_1,\ldots, \psi_u \in
\mathcal{L}^\Theta_x(\mathbb{M})$ and $\vee \in \mathbb{M}$, we
also get that $\bigvee^u_{i = 1}\psi_i \in
\mathcal{L}^\Theta_x(\mathbb{M})$.

 We have then
$$
b_{m+1} \not\models_t \bigvee^u_{i = 1}\mu^-(\psi_i,\ldots,
\psi_i),
$$
whence by Lemma \ref{L:modalities}.4 we get that
$$
b_{m+1} \not\models_t \mu^-(\bigvee^u_{i = 1}\psi_i,\ldots,
\bigvee^u_{i = 1}\psi_i),
$$
and further, that
$$
b_1 \not\models_t \mu(\bigvee^u_{i = 1}\psi_i,\ldots, \bigvee^u_{i
= 1}\psi_i).
$$
Therefore, by $a_1\mathrel{B}b_1$ and the fact that $\bigvee^u_{i
= 1}\psi_i \in \mathcal{L}^\Theta_x(\mathbb{M})$, we infer that
$$a_1 \not\models_r \mu(\bigvee^u_{i = 1}\psi_i,\ldots,
\bigvee^u_{i = 1}\psi_i),
$$
thus obtaining that there must be $a_2,\ldots, a_{m + 1} \in U_r$,
such that we have $\pi^m_1S^ra$ and, moreover:
$$a_{m + 1} \not\models_r \mu^-(\bigvee^u_{i = 1}\psi_i,\ldots,
\bigvee^u_{i = 1}\psi_i).
$$
Whence, again by Lemma \ref{L:modalities}.4, we get that
$$
a_{m+1} \not\models_r \bigvee^u_{i = 1}\mu^-(\psi_i,\ldots,
\psi_i),
$$
This, in turn, means that the set of formulas
$$
\{ S_1(a_1, x_2), \pi^m_2Sx \} \cup \{ \neg\mu^-(\psi_1,\ldots,
\psi_1),\ldots, \neg\mu^-(\psi_u,\ldots, \psi_u) \}
$$
is satisfiable at $M_r/a_1$. But since the set in \eqref{E:set1}
was chosen as an arbitrary subset of $\mathbb{F}$, we have that
every finite subset of the set
$$
\{ S_1(a_1, x_2), \pi^m_2Sx \} \cup \{ \neg\psi(x_{m+1}) \mid \psi
\in \mathbb{F} \}
$$
is satisfiable at  $M_r/a_1$. Therefore, by compactness of
first-order logic, this set is consistent with $Th(M_r/a_1)$ and,
by $\omega$-saturation of both $M_1$ and $M_2$, it must be
satisfied in $M_r/a_1$ by some $a_2,\ldots, a_{m+1}  \in U_r$. So
for any such $a_2,\ldots, a_{m+1}$ we will have both $\pi^m_1S^ra$
and, moreover
$$
a_{m+ 1} \models_r \{ \neg\psi(x_{m+1}) \mid \psi \in \mathbb{F}
\}.
$$

Thus, by independence of truth at a sequence of elements from the
choice of free variables in a formula, we will also have
$$
a_{m+1} \models_r \{ \neg\psi(x) \mid \psi \in \mathbb{F} \}.
$$
This means that
$$
\{ \langle a_{m+1},b_{m+1}\rangle \} \in Rel(\{
\mu^-(\psi,\ldots,\psi)\mid \psi \in
\mathcal{L}^\Theta_x(\mathbb{M}) \},M_1, M_2),
$$
and therefore, by Lemma \ref{L:modalities}.1 we get that
$$
\{ \langle a_{m+1},b_{m+1}\rangle \} \in Rel(\{
\mu^-(\psi_1,\ldots,\psi_n)\mid \psi_1,\ldots,\psi_n \in
\mathcal{L}^\Theta_x(\mathbb{M}) \},M_1, M_2).
$$
Whence by definition of $A_2$ we get that
$a_{m+1}\mathrel{A_2}b_{m+1}$, and thus that condition
\eqref{E:cond-A} for $\langle C,A_2, B\rangle$ is satisfied. So we
conclude that $\langle C,A_2, B\rangle \in [\mu](\{ A \})$.

\emph{Case 2.2}. $\mu \in \nu_x(\forall\exists, f)$, where $f$ is
a non-constant Boolean anti-monotone function. This case is
similar to the previous, the difference being that instead of
Lemma \ref{L:modalities}.4 one has to apply \ref{L:modalities}.5.

\emph{Case 2.3}. $\mu \in \nu_x(\exists\forall, f)$, where $f$ is
a non-constant Boolean monotone function.

Then we can assume that $\mu$ has a form $\exists x_2\ldots x_{m +
1}(\pi^{m}_1Sx \wedge \mu^-)$ in the assumptions of
\eqref{E:Eguard} and that $\mu^- \in \nu_x(\forall, f)$.

Consider $A_2$ as defined in lemma. We have of course

\begin{equation}\label{E:A23}
A_2 \in Rel(\{ \mu^-(\psi_1,\ldots, \psi_n)\mid \psi_1,\ldots,
\psi_n \in \mathcal{L}^\Theta_x(\mathbb{M}) \},M_1, M_2).
\end{equation}

Therefore, by induction hypothesis, for some $C \in \{ A \} \cup
[\mu^0](\{ A \})$ the couple $\langle C,A_2\rangle$ is in
$[\mu^-](\{ A \})$.

Now assume that
$$
B \in Rel(\{ \mu(\psi_1,\ldots, \psi_n)\mid \psi_1,\ldots, \psi_n
\in \mathcal{L}^\Theta_x(\mathbb{M}) \},M_1, M_2).
$$
To show that $\langle C,A_2, B\rangle$ is in $[\mu](\{ A \})$, we
only need to verify condition \eqref{E:cond-E}.

So assume that $\{ r,t \} = \{ 1,2 \}$ and that we have
$\bar{a}_{m+1} \in U_r$, $b_1 \in U_t$, and $a_1\mathrel{B}b_1$.
Moreover, assume that $\pi^m_1S^ra$. Then consider the set
$$
\mathbb{T} = Tr(\{ \mu^-(\psi,\ldots,\psi)\mid \psi \in
\mathcal{L}^\Theta_{x_{m+1}}(\mathbb{M}) \}, M_r, a_{m+1}).
$$
This set is non-empty, since we have $\top \in
\mathcal{L}^\Theta_{x_{m+1}}(\mathbb{M})$, and, further:
$$
\top \Leftrightarrow \mu^-(\top,\ldots,\top) \in \mathbb{T}.
$$
Now, take an arbitrary finite subset

\begin{equation}\label{E:set3}
\{ \mu^-(\psi_1,\ldots, \psi_1),\ldots, \mu^-(\psi_u,\ldots,
\psi_u) \} \subseteq \mathbb{T}.
\end{equation}
Note that since we have $\psi_1,\ldots, \psi_u \in
\mathcal{L}^\Theta_x(\mathbb{M})$ and $\wedge \in \mathbb{M}$, we
also get that $\bigwedge^u_{i = 1}\psi_i \in
\mathcal{L}^\Theta_x(\mathbb{M})$.

 We have then
$$
a_{m+1} \models_r \bigwedge^u_{i = 1}\mu^-(\psi_i,\ldots, \psi_i),
$$
whence by Lemma \ref{L:modalities}.2 we get that
$$
a_{m+1} \models_r \mu^-(\bigwedge^u_{i = 1}\psi_i,\ldots,
\bigwedge^u_{i = 1}\psi_i),
$$
and further, that
$$
a_1 \models_r \mu(\bigwedge^u_{i = 1}\psi_i,\ldots, \bigwedge^u_{i
= 1}\psi_i).
$$
Therefore, by $a_1\mathrel{B}b_1$ and the fact that
$\bigwedge^u_{i = 1}\psi_i \in \mathcal{L}^\Theta_x(\mathbb{M})$,
we infer that
$$b_1 \models_t \mu(\bigwedge^u_{i = 1}\psi_i,\ldots,
\bigwedge^u_{i = 1}\psi_i),
$$
thus obtaining that there must be $b_2,\ldots, b_{m + 1} \in U_t$,
such that we have $\pi^m_1S^tb$ and, moreover:
$$b_{m + 1} \models_t \mu^-(\bigwedge^u_{i = 1}\psi_i,\ldots,
\bigwedge^u_{i = 1}\psi_i).
$$
Whence, again by Lemma \ref{L:modalities}.2, we get that
$$
b_{m+1} \models_t \bigwedge^u_{i = 1}\mu^-(\psi_i,\ldots, \psi_i),
$$
This, in turn, means that the set of formulas
$$
\{ S_1(a_1, x_2), \pi^m_2Sx \} \cup \{ \mu^-(\psi_1,\ldots,
\psi_1),\ldots, \mu^-(\psi_u,\ldots, \psi_u) \}
$$
is satisfiable at $M_t/b_1$. But since the set in \eqref{E:set3}
was chosen as an arbitrary subset of $\mathbb{T}$, we have that
every finite subset of the set
$$
\{ S_1(a_1, x_2), \pi^m_2Sx \} \cup \mathbb{T}
$$
is satisfiable at  $M_t/b_1$. Therefore, by compactness of
first-order logic, this set is consistent with $Th(M_t/b_1)$ and,
by $\omega$-saturation of both $M_1$ and $M_2$, it must be
satisfied in $M_t/b_1$ by some $b_2,\ldots, b_{m+1}  \in U_t$. So
for any such $b_2,\ldots, b_{m+1}$ we will have both $\pi^m_1S^tb$
and, moreover
$$
b_{m+ 1} \models_t \mathbb{T}.
$$

Thus, by independence of truth at a sequence of elements from the
choice of free variables in a formula, we will also have
$$
b_{m+1} \models_t \{ \psi(x) \mid \psi \in \mathbb{T} \}.
$$
This means that
$$
\{ \langle a_{m+1},b_{m+1}\rangle \} \in Rel(\{
\mu^-(\psi,\ldots,\psi)\mid \psi \in
\mathcal{L}^\Theta_x(\mathbb{M}) \},M_1, M_2),
$$
and therefore, by Lemma \ref{L:modalities}.1 we get that
$$
\{ \langle a_{m+1},b_{m+1}\rangle \} \in Rel(\{
\mu^-(\psi_1,\ldots,\psi_n)\mid \psi_1,\ldots,\psi_n \in
\mathcal{L}^\Theta_x(\mathbb{M}) \},M_1, M_2).
$$
Whence by the definition of $A_2$ we get that
$a_{m+1}\mathrel{A_2}b_{m+1}$, and thus that condition
\eqref{E:cond-E} for $\langle C,A_2, B\rangle$ is satisfied. So we
conclude that $\langle C,A_2, B\rangle \in [\mu](\{ A \})$.

\emph{Case 2.4}. $\mu \in \nu_x(\exists\forall, f)$, where $f$ is
a non-constant Boolean anti-monotone function. This case is
similar to the previous, the difference being that instead of
Lemma \ref{L:modalities}.2 one has to apply \ref{L:modalities}.3.

\emph{Case 2.5}. $\mu \in \nu_x(\forall\exists, f)$, where $f$ is
a rest non-$FTF$-function. Then, as we have shown in Section
\ref{S:bool}, $f$ is a $TFT$-function. Further, we can assume that
$\mu$ has a form
$$
\forall x_2\ldots x_{m_1 + 1}(\pi^{m_1}_1Sx \to \exists x_{m_1 +
2}\ldots x_{m_1 + m_2 + 1}(\pi^{m_1 + m_2}_{m_1 + 1}Sx \wedge
\phi(P_1(x_{m_1 + m_2 + 1}),\ldots,P_n(x_{m_1 + m_2 + 1}))))
$$
 in the
assumptions of \eqref{E:Aguard} and \eqref{E:Eguard}, where $\phi$
defines $f$ for $P_1(x_{m_1 + m_2 + 1}),\ldots,P_n(x_{m_1 + m_2 +
1})$, and that $\mu^-$ has the form:
$$
\exists x_{m_1 + 2}\ldots x_{m_1 + m_2 + 1}(\pi^{m_1 + m_2}_{m_1 +
1}Sx \wedge \phi(P_1(x_{m_1 + m_2 + 1}),\ldots,P_n(x_{m_1 + m_2 +
1})))
$$
under the same assumptions.

Consider $A_2$ as defined in lemma. We have of course

\begin{equation}\label{E:A25}
A_2 \in Rel(\{ \mu^-(\psi_1,\ldots, \psi_n)\mid \psi_1,\ldots,
\psi_n \in \mathcal{L}^\Theta_x(\mathbb{M}) \},M_1, M_2).
\end{equation}

Therefore, by induction hypothesis, for some $C \in \{ A \} \cup
[\mu^0](\{ A \})$ the couple $\langle C,A_2\rangle$ is in
$[\mu^-](\{ A \})$.

Now assume that
$$
B \in Rel(\{ \mu(\psi_1,\ldots, \psi_n)\mid \psi_1,\ldots, \psi_n
\in \mathcal{L}^\Theta_x(\mathbb{M}) \},M_1, M_2).
$$
To show that $\langle C,A_2, B\rangle$ is in $[\mu](\{ A \})$, we
only need to verify condition \eqref{E:cond-A}.

So assume that $\{ r,t \} = \{ 1,2 \}$ and that we have $a_1 \in
U_r$, $\bar{b}_{m+1} \in U_t$, and $a_1\mathrel{B}b_1$. Moreover,
assume that $\pi^m_1S^tb$. Then consider the set
$$
\mathbb{F} = Fa(\{ \mu^-(\psi_1,\ldots,\psi_n)\mid
\psi_1,\ldots,\psi_n \in \mathcal{L}^\Theta_{x_{m_1 + m_2 +
1}}(\mathbb{M}) \}, M_t, b_{m+1}).
$$
This set is non-empty, since we have $\top, \bot \in
\mathcal{L}^\Theta_{x_{m_1 + m_2 + 1}}(\mathbb{M})$, and $\phi$
above defines a rest Boolean function, hence a function which is
false for at least one $n$-tuple of Boolean values. Therefore, we
get:
$$
\bot \Leftrightarrow \mu^-(Bool_1,\ldots,Bool_n) \in \mathbb{F},
$$
for appropriate $Bool_1,\ldots,Bool_n \in \{ \top, \bot \}$. Now,
take an arbitrary finite subset

\begin{equation}\label{E:set5}
\{ \mu^-(\psi^1_1,\ldots, \psi^1_n),\ldots, \mu^-(\psi^u_1,\ldots,
\psi^u_n) \} \subseteq \mathbb{F}.
\end{equation}

 We have then

\begin{equation}\label{E:star1}
b_{m+1} \not\models_t \bigvee^u_{i = 1}\mu^-(\psi^i_1,\ldots,
\psi^i_n)
\end{equation}.

Consider the latter disjunction. We can build for it the following
chain of logical equivalents. First, using the definition of
$\mu^-$ and the distributivity of $\exists$ over $\vee$ we
transform it into
$$
\exists x_{m_1 + 2}\ldots x_{m_1 + m_2 + 1}(\pi^{m_1 + m_2}_{m_1 +
1}Sx \wedge \bigvee^u_{i = 1} \phi(\psi^i_1(x_{m_1 + m_2 +
1}),\ldots,\psi^i_n(x_{m_1 + m_2 + 1}))).
$$
Further, using Lemma \ref{L:Boolean}.8, we get that the latter
formula is equivalent to
$$
\exists x_{m_1 + 2}\ldots x_{m_1 + m_2 + 1}(\pi^{m_1 + m_2}_{m_1 +
1}Sx \wedge \bigvee^u_{i = 1}(K^i_1 \vee,\ldots, \vee K^i_{i_k}
\vee L^i_1 \vee,\ldots, \vee L^i_{i_m})),
$$
where $K$'s are conjunctions of $\psi$'s and $L$'s are
conjunctions of negated $\psi$'s. Then, using the laws of
classical propositional logic we get the following logical
equivalents for our formula:
$$
\exists x_{m_1 + 2}\ldots x_{m_1 + m_2 + 1}(\pi^{m_1 + m_2}_{m_1 +
1}Sx \wedge \bigvee^u_{i = 1}(K^i_1 \vee,\ldots, \vee K^i_{i_k})
\vee \bigvee^u_{i = 1}(L^i_1 \vee,\ldots, \vee L^i_{i_m})),
$$
and further, by De Morgan laws we push the negations out of $L$'s,
getting:
$$
\exists x_{m_1 + 2}\ldots x_{m_1 + m_2 + 1}(\pi^{m_1 + m_2}_{m_1 +
1}Sx \wedge \bigvee^u_{i = 1}(K^i_1 \vee,\ldots, \vee K^i_{i_k})
\vee \bigvee^u_{i = 1}(\neg\tilde{L}^i_1 \vee,\ldots, \vee
\neg\tilde{L}^i_{i_m})),
$$
where the $\tilde{L}$'s are the respective disjunctions of
$\psi$'s. Further applications of De Morgan laws yield:
$$
\exists x_{m_1 + 2}\ldots x_{m_1 + m_2 + 1}(\pi^{m_1 + m_2}_{m_1 +
1}Sx \wedge \bigvee^u_{i = 1}(K^i_1 \vee,\ldots, \vee K^i_{i_k})
\vee \bigvee^u_{i = 1}\neg(\tilde{L}^i_1 \wedge,\ldots, \wedge
\tilde{L}^i_{i_m})),
$$
and
$$
\exists x_{m_1 + 2}\ldots x_{m_1 + m_2 + 1}(\pi^{m_1 + m_2}_{m_1 +
1}Sx \wedge \bigvee^u_{i = 1}(K^i_1 \vee,\ldots, \vee K^i_{i_k})
\vee \neg\bigwedge^u_{i = 1}(\tilde{L}^i_1 \wedge,\ldots, \wedge
\tilde{L}^i_{i_m})).
$$
 We now set
$$
\Phi := \bigvee^u_{i = 1}(K^i_1 \vee,\ldots, \vee K^i_{i_k}),
$$
and
$$
\Psi := \bigwedge^u_{i = 1}(\tilde{L}^i_1 \wedge,\ldots, \wedge
\tilde{L}^i_{i_m}),
$$
thus getting the next logical equivalent to our formula in the
following form:
$$
\exists x_{m_1 + 2}\ldots x_{m_1 + m_2 + 1}(\pi^{m_1 + m_2}_{m_1 +
1}Sx \wedge (\Phi \vee \neg\Psi)).
$$

 Note that since all the $\psi$'s, by their choice,
are in $\mathcal{L}^\Theta_{x_{m_1 + m_2 + 1}}(\mathbb{M})$ and we
have $\wedge, \vee \in \mathbb{M}$, we also get that $\Phi, \Psi
\in \mathcal{L}^\Theta_{x_{m_1 + m_2 + 1}}(\mathbb{M})$. Moreover,
since $\phi$ defines a $TFT$ function, then, by Lemma
\ref{L:Boolean}.5, there exist $\tau_1, \ldots, \tau_n \in \{
\Phi, \Psi, \Phi \vee \Psi, \top, \bot \}$, for which we have the
last formula in the above series of logical equivalents equivalent
to

$$
\exists x_{m_1 + 2}\ldots x_{m_1 + m_2 + 1}(\pi^{m_1 + m_2}_{m_1 +
1}Sx \wedge \phi(\tau_1,\ldots,\tau_n)),
$$

and further, to

$$
\mu^-(\tau_1,\ldots,\tau_n).
$$
By transitivity of logical equivalence, we get that
\eqref{E:star1} is equivalent to
$$
b_{m+1} \not\models_t \mu^-(\tau_1,\ldots,\tau_n),
$$

whence, by definition of $\mu$ and the choice of $b_{m+1}$ we
infer that
$$
b_1 \not\models_t \mu(\tau_1,\ldots,\tau_n).
$$
Therefore, by $a_1\mathrel{B}b_1$ and the fact that
$\tau_1,\ldots,\tau_n \in \mathcal{L}^\Theta_{x_{m_1 + m_2 +
1}}(\mathbb{M})$, we get that
$$
a_1 \not\models_r \mu(\tau_1,\ldots,\tau_n),
$$
thus obtaining that there must be $a_2,\ldots, a_{m + 1} \in U_r$,
such that we have $\pi^m_1S^ra$ and, moreover:
$$a_{m + 1} \not\models_r \mu^-(\tau_1,\ldots,\tau_n).
$$
Whence, using the above chain of logical equivalents in the
reverse direction, we get that
$$
a_{m+1} \not\models_r\bigvee^u_{i = 1}\mu^-(\psi^i_1,\ldots,
\psi^i_n).
$$
This, in turn, means that the set of formulas
$$
\{ S_1(a_1, x_2), \pi^m_2Sx \} \cup \{ \neg \mu^-(\psi^1_1,\ldots,
\psi^1_n),\ldots, \neg\mu^-(\psi^u_1,\ldots, \psi^u_n) \}
$$
is satisfiable at $M_r/a_1$. But since the set in \eqref{E:set5}
was chosen as an arbitrary subset of $\mathbb{F}$, we have that
every finite subset of the set
$$
\{ S_1(a_1, x_2), \pi^m_2Sx \} \cup \{ \neg \mu^-(\psi_1,\ldots,
\psi_n) \mid \mu^-(\psi_1,\ldots, \psi_n) \in \mathbb{F} \}
$$
is satisfiable at  $M_r/a_1$. Therefore, by compactness of
first-order logic, this set is consistent with $Th(M_r/a_1)$ and,
by $\omega$-saturation of both $M_1$ and $M_2$, it must be
satisfied in $M_r/a_1$ by some $a_2,\ldots, a_{m+1}  \in U_r$. So
for any such $a_2,\ldots, a_{m+1}$ we will have both $\pi^m_1S^ra$
and, moreover
$$
a_{m+ 1} \models_r \{ \neg \mu^-(\psi_1,\ldots, \psi_n) \mid
\mu^-(\psi_1,\ldots, \psi_n) \in \mathbb{F} \}.
$$

Thus, by independence of truth at a sequence of elements from the
choice of free variables in a formula, we will also have
$$
a_{m+1} \models_r \{ \neg \mu^-(\psi_1,\ldots, \psi_n)(x) \mid
\mu^-(\psi_1,\ldots, \psi_n) \in \mathbb{F} \}.
$$
This means that
$$
\{ \langle a_{m+1},b_{m+1}\rangle \} \in Rel(\{
\mu^-(\psi_1,\ldots,\psi_n)\mid \psi \in
\mathcal{L}^\Theta_x(\mathbb{M}) \},M_1, M_2).
$$
Whence by the definition of $A_2$ we get that
$a_{m+1}\mathrel{A_2}b_{m+1}$, and thus that condition
\eqref{E:cond-A} for $\langle C,A_2, B\rangle$ is satisfied. So we
conclude that $\langle C,A_2, B\rangle \in [\mu](\{ A \})$.

\emph{Case 2.6}. $\mu \in \nu_x(\exists\forall, f)$, where $f$ is
a rest non-$TFT$-function. Then, as we have shown in Section
\ref{S:bool}, $f$ is an $FTF$-function. Further, we can assume
that $\mu$ has a form
$$
\exists x_2\ldots x_{m_1 + 1}(\pi^{m_1}_1Sx \wedge \forall x_{m_1
+ 2}\ldots x_{m_1 + m_2 + 1}(\pi^{m_1 + m_2}_{m_1 + 1}Sx \to
\phi(P_1(x_{m_1 + m_2 + 1}),\ldots,P_n(x_{m_1 + m_2 + 1}))))
$$
 in the
assumptions of \eqref{E:Aguard} and \eqref{E:Eguard}, where $\phi$
defines $f$ for $P_1(x_{m_1 + m_2 + 1}),\ldots,P_n(x_{m_1 + m_2 +
1})$, and that $\mu^-$ has the form:
$$
\forall x_{m_1 + 2}\ldots x_{m_1 + m_2 + 1}(\pi^{m_1 + m_2}_{m_1 +
1}Sx \to \phi(P_1(x_{m_1 + m_2 + 1}),\ldots,P_n(x_{m_1 + m_2 +
1})))
$$
under the same assumptions.

Consider $A_2$ as defined in lemma. We have of course

\begin{equation}\label{E:A26}
A_2 \in Rel(\{ \mu^-(\psi_1,\ldots, \psi_n)\mid \psi_1,\ldots,
\psi_n \in \mathcal{L}^\Theta_x(\mathbb{M}) \},M_1, M_2).
\end{equation}

Therefore, by induction hypothesis, for some $C \in \{ A \} \cup
[\mu^0](\{ A \})$ the couple $\langle C,A_2\rangle$ is in
$[\mu^-](\{ A \})$.

Now assume that
$$
B \in Rel(\{ \mu(\psi_1,\ldots, \psi_n)\mid \psi_1,\ldots, \psi_n
\in \mathcal{L}^\Theta_x(\mathbb{M}) \},M_1, M_2).
$$
To show that $\langle C,A_2, B\rangle$ is in $[\mu](\{ A \})$, we
only need to verify condition \eqref{E:cond-E}.

So assume that $\{ r,t \} = \{ 1,2 \}$ and that we have
$\bar{a}_{m+1} \in U_r$, $b_1 \in U_t$, and $a_1\mathrel{B}b_1$.
Moreover, assume that $\pi^m_1S^ra$. Then consider the set
$$
\mathbb{T} = Tr(\{ \mu^-(\psi_1,\ldots,\psi_n)\mid
\psi_1,\ldots,\psi_n \in \mathcal{L}^\Theta_{x_{m_1 + m_2 +
1}}(\mathbb{M}) \}, M_r, a_{m+1}).
$$
This set is non-empty, since we have $\top, \bot \in
\mathcal{L}^\Theta_{x_{m_1 + m_2 + 1}}(\mathbb{M})$, and $\phi$
above defines a rest Boolean function, hence a function which is
true for at least one $n$-tuple of Boolean values. Therefore, we
get:
$$
\top \Leftrightarrow \mu^-(Bool_1,\ldots,Bool_n) \in \mathbb{F},
$$
for appropriate $Bool_1,\ldots,Bool_n \in \{ \top, \bot \}$. Now,
take an arbitrary finite subset

\begin{equation}\label{E:set6}
\{ \mu^-(\psi^1_1,\ldots, \psi^1_n),\ldots, \mu^-(\psi^u_1,\ldots,
\psi^u_n) \} \subseteq \mathbb{T}.
\end{equation}

 We have then

\begin{equation}\label{E:star2}
a_{m+1} \models_r \bigwedge^u_{i = 1}\mu^-(\psi^i_1,\ldots,
\psi^i_n)
\end{equation}.

Consider the latter conjunction. We can build for it the following
chain of logical equivalents. First, using the definition of
$\mu^-$ and the distributivity of $\forall$ over $\wedge$ we
transform it into
$$
\forall x_{m_1 + 2}\ldots x_{m_1 + m_2 + 1}(\pi^{m_1 + m_2}_{m_1 +
1}Sx \to \bigwedge^u_{i = 1} \phi(\psi^i_1(x_{m_1 + m_2 +
1}),\ldots,\psi^i_n(x_{m_1 + m_2 + 1}))).
$$
Further, using Lemma \ref{L:Boolean}.9, we get that the latter
formula is equivalent to
$$
\forall x_{m_1 + 2}\ldots x_{m_1 + m_2 + 1}(\pi^{m_1 + m_2}_{m_1 +
1}Sx \to \bigwedge^u_{i = 1}(K^i_1 \wedge,\ldots, \wedge K^i_{i_k}
\wedge L^i_1 \wedge,\ldots, \wedge L^i_{i_m})),
$$
where $K$'s are disjunctions of $\psi$'s and $L$'s are
disjunctions of negated $\psi$'s. Then, using the laws of
classical propositional logic we get the following logical
equivalents for our formula:
$$
\forall x_{m_1 + 2}\ldots x_{m_1 + m_2 + 1}(\pi^{m_1 + m_2}_{m_1 +
1}Sx \to \bigwedge^u_{i = 1}(K^i_1 \wedge,\ldots, \wedge
K^i_{i_k}) \wedge \bigwedge^u_{i = 1}(L^i_1 \wedge,\ldots, \wedge
L^i_{i_m})),
$$
and further, by De Morgan laws we push the negations out of $L$'s,
getting:
$$
\forall x_{m_1 + 2}\ldots x_{m_1 + m_2 + 1}(\pi^{m_1 + m_2}_{m_1 +
1}Sx \to \bigwedge^u_{i = 1}(K^i_1 \wedge,\ldots, \wedge
K^i_{i_k}) \wedge \bigwedge^u_{i = 1}(\neg\tilde{L}^i_1
\wedge,\ldots, \wedge \neg\tilde{L}^i_{i_m})),
$$
where the $\tilde{L}$'s are the respective conjunctions of
$\psi$'s. Further applications of De Morgan laws yield:
$$
\forall x_{m_1 + 2}\ldots x_{m_1 + m_2 + 1}(\pi^{m_1 + m_2}_{m_1 +
1}Sx \to \bigwedge^u_{i = 1}(K^i_1 \wedge,\ldots, \wedge
K^i_{i_k}) \wedge \bigwedge^u_{i = 1}\neg(\tilde{L}^i_1
\vee,\ldots, \vee \tilde{L}^i_{i_m})),
$$
and
$$
\forall x_{m_1 + 2}\ldots x_{m_1 + m_2 + 1}(\pi^{m_1 + m_2}_{m_1 +
1}Sx \to \bigwedge^u_{i = 1}(K^i_1 \wedge,\ldots, \wedge
K^i_{i_k}) \wedge \neg\bigvee^u_{i = 1}(\tilde{L}^i_1 \vee,\ldots,
\vee \tilde{L}^i_{i_m})).
$$
We now set
$$
\Phi := \bigwedge^u_{i = 1}(K^i_1 \wedge,\ldots, \wedge
K^i_{i_k}),
$$
and
$$
\Psi := \bigvee^u_{i = 1}(\tilde{L}^i_1 \vee,\ldots, \vee
\tilde{L}^i_{i_m}),
$$
thus getting the next logical equivalent to our formula in the
following form:
$$
\forall x_{m_1 + 2}\ldots x_{m_1 + m_2 + 1}(\pi^{m_1 + m_2}_{m_1 +
1}Sx \to (\Phi \wedge \neg\Psi)).
$$

 Note that since all the $\psi$'s, by their choice,
are in $\mathcal{L}^\Theta_{x_{m_1 + m_2 + 1}}(\mathbb{M})$ and we
have $\wedge, \vee \in \mathbb{M}$, we also get that $\Phi, \Psi
\in \mathcal{L}^\Theta_{x_{m_1 + m_2 + 1}}(\mathbb{M})$. Moreover,
since $\phi$ defines an $FTF$ function, then, by Lemma
\ref{L:Boolean}.6, there exist $\tau_1, \ldots, \tau_n \in \{
\Phi, \Psi, \Phi \wedge \Psi, \top, \bot \}$, for which we have
the last formula in the above series of logical equivalents
equivalent to

$$
\forall x_{m_1 + 2}\ldots x_{m_1 + m_2 + 1}(\pi^{m_1 + m_2}_{m_1 +
1}Sx \to \phi(\tau_1,\ldots,\tau_n)),
$$

and further, to

$$
\mu^-(\tau_1,\ldots,\tau_n).
$$
Therefore, we get that \eqref{E:star2} is equivalent to
$$
a_{m+1} \models_r \mu^-(\tau_1,\ldots,\tau_n),
$$

whence, by definition of $\mu$ and the choice of $a_{m+1}$ we
infer that
$$
a_1 \models_r \mu(\tau_1,\ldots,\tau_n).
$$
Therefore, by $a_1\mathrel{B}b_1$ and the fact that
$\tau_1,\ldots,\tau_n \in \mathcal{L}^\Theta_x(\mathbb{M})$, we
get that
$$
b_1 \models_t \mu(\tau_1,\ldots,\tau_n),
$$
thus obtaining that there must be $b_2,\ldots, b_{m + 1} \in U_t$,
such that we have $\pi^m_1S^tb$ and, moreover:
$$b_{m + 1} \models_t \mu^-(\tau_1,\ldots,\tau_n).
$$
Whence, using the above chain of logical equivalents in the
reverse direction, we get that
$$
b_{m+1} \models_t\bigwedge^u_{i = 1}\mu^-(\psi^i_1,\ldots,
\psi^i_n).
$$
This, in turn, means that the set of formulas
$$
\{ S_1(a_1, x_2), \pi^m_2Sx \} \cup \{ \mu^-(\psi^1_1,\ldots,
\psi^1_n),\ldots, \mu^-(\psi^u_1,\ldots, \psi^u_n) \}
$$
is satisfiable at $M_t/b_1$. But since the set in \eqref{E:set6}
was chosen as an arbitrary subset of $\mathbb{T}$, we have that
every finite subset of the set
$$
\{ S_1(a_1, x_2), \pi^m_2Sx \} \cup \mathbb{T}
$$
is satisfiable at  $M_t/b_1$. Therefore, by compactness of
first-order logic, this set is consistent with $Th(M_t/b_1)$ and,
by $\omega$-saturation of both $M_1$ and $M_2$, it must be
satisfied in $M_t/b_1$ by some $b_2,\ldots, b_{m+1}  \in U_t$. So
for any such $b_2,\ldots, b_{m+1}$ we will have both $\pi^m_1S^tb$
and, moreover
$$
b_{m+ 1} \models_t \mathbb{T}.
$$

Thus, by independence of truth at a sequence of elements from the
choice of free variables in a formula, we will also have
$$
a_{m+1} \models_t \{ \mu^-(\psi_1,\ldots, \psi_n)(x) \mid
\mu^-(\psi_1,\ldots, \psi_n) \in \mathbb{T} \}.
$$
This means that
$$
\{ \langle a_{m+1},b_{m+1}\rangle \} \in Rel(\{
\mu^-(\psi_1,\ldots,\psi_n)\mid \psi \in
\mathcal{L}^\Theta_x(\mathbb{M}) \},M_1, M_2).
$$
Whence by the definition of $A_2$ we get that
$a_{m+1}\mathrel{A_2}b_{m+1}$, and thus that condition
\eqref{E:cond-A} for $\langle C,A_2, B\rangle$ is satisfied. So we
conclude that $\langle C,A_2, B\rangle \in [\mu](\{ A \})$.
\end{proof}

\begin{corollary}\label{L:sat}
Let $\mathcal{L}^\Theta_x(\mathbb{M})$ be a standard $x$-fragment
of the correspondence language, such that $\mathbb{M} = \{
\mu_1,\ldots, \mu_s \} \supseteq \{ \wedge, \vee, \top, \bot \}$,
and let $M_1, M_2$ be saturated models. Then binary relation $A
\in W(M_1,M_2)$ such that

$$
A = \bigcup Rel(\mathcal{L}^\Theta_x(\mathbb{M}),M_1, M_2)
$$
is an $(\mathcal{L}^\Theta_x(\mathbb{M}), M_1, M_2)$-asimulation,
whenever $A$ is non-empty.
\end{corollary}

\begin{proof}
By conditions of the lemma we have that $A$ is non-empty and that
$A \in Rel(\mathcal{L}_x(\emptyset),M_1, M_2)$. So we only need to
show that for every $i$ such that $1 \leq i \leq s$ there exist
$A_1,\ldots, A_{\delta(\mu_i)}$ such that
$$
\langle A_1,\ldots, A_{\delta(\mu_i)}, A\rangle \in [\mu](\{ A
\}).
$$
Well, given that by the assumption of the lemma we clearly have
$$
A \in Rel(\mathcal{L}^\Theta_x(\mathbb{M}),M_1, M_2)
$$
then, assuming $\delta(\mu_i) > 0$, it follows from Lemma
\ref{L:mu-sat} that if $\mu_i^0,\ldots, \mu_i^{\delta(\mu_i) - 1}$
is the set of ancestors of $\mu$ and for $1 \leq j \leq
\delta(\mu_i)$ binary relation $A_j \in W(M_1, M_2)$ is such that
$$
A_j = \bigcup Rel(\{ \mu^{j-1}(\psi_1,\ldots, \psi_n)\mid
\psi_1,\ldots \psi_n \in \mathcal{L}^\Theta_x(\mathbb{M}) \},M_1,
M_2),
$$
then for some $C \in \{ A \} \cup [\mu^0](\{ A \})$ the tuple
$\langle C,A_2,\ldots, A_{\delta(\mu)}, A\rangle$ is in $[\mu](\{
A \})$.

On the other hand, if $\delta(\mu_i) = 0$, then if $\mu_i$ defines
$\top$ or $\bot$, we certainly have $A \in W(M_1, M_2) =
[\mu_i](A)$. If $\mu_i$ defines a non-constant Boolean monotone
function, then we have $A \in \{ A \} = [\mu_i](A)$. Finally, if
$\mu_i$ defines either a non-constant anti-monotone function or a
rest function, then, in presence of $\wedge, \vee, \top$, and
$\bot$, one can show, using respectively either Lemma
\ref{L:Boolean}.4 or Lemma \ref{L:Boolean}.7, that for every $\psi
\in \mathcal{L}^\Theta_x(\mathbb{M})$, formula $\neg\psi$ is also
in $\mathcal{L}^\Theta_x(\mathbb{M})$. Therefore, we have:
$$
A = \bigcup Rel(\mathcal{L}^\Theta_x(\mathbb{M}),M_1, M_2) =
\bigcup Rel(\{\neg\psi\mid \psi\in\mathcal{L}^\Theta_x(\mathbb{M})
\},M_1, M_2) = A^{-1},
$$
so that we get
$$
A = A^{-1} = A \cap A^{-1}.
$$
But, since $[\mu_i](A)$ is either $\{ A^{-1} \}$ or $\{ A \cap
A^{-1} \}$, in both cases we get $A \in [\mu_i](A)$.
\end{proof}

\section{The main result}\label{S:Main}

Thus far we have only dealt with asimulations ``locally'', i.e. as
subsets of $W(M_1,M_2)$. We now give the global definition of
asimulation as a class of binary relations:

\begin{definition}
$1$. Let $\mathcal{L}^\Theta_x(\mathbb{M})$ be a standard
fragment. Then the class
$A\sigma(\mathcal{L}^\Theta_x(\mathbb{M}))$ defined as follows:

$$
A\sigma(\mathcal{L}^\Theta_x(\mathbb{M})) := \bigcup\{
A\sigma(\mathcal{L}^\Theta_x(\mathbb{M}), M_1, M_2)\mid M_1, M_2
\text{ are }\Theta\text{-models}\}
$$
is called the class of
$\mathcal{L}^\Theta_x(\mathbb{M})$-asimulations.

$2$. We say that a formula $\varphi(x)$ is invariant w.r.t.
$\mathcal{L}^\Theta_x(\mathbb{M})$-asimulations iff it is
invariant w.r.t. to the class
$A\sigma(\mathcal{L}^\Theta_x(\mathbb{M}))$.
\end{definition}

We are now able to formulate and prove our main result:
\begin{theorem}\label{L:final}
Let $\varphi(x)$ be a formula in the correspondence language such
that $\Sigma_\varphi \subseteq \Theta$ and let $\{ \wedge, \vee,
\bot,\top \} \subseteq \mathbb{M}$. Then $\varphi(x)$ is logically
equivalent to a formula in a standard fragment
$\mathcal{L}^\Theta_x(\mathbb{M})$ iff $\varphi(x)$ is invariant
w.r.t. $\mathcal{L}^\Theta_x(\mathbb{M})$-asimulations.
\end{theorem}
\begin{proof}
The left-to-right direction of the Theorem immediately follows
from Corollary \ref{L:cor}. We consider the other direction.

Assume that $\varphi(x)$ is invariant w.r.t.
$\mathcal{L}^\Theta_x(\mathbb{M})$-asimulations. We may assume
that $\varphi(x)$ is satisfiable, for $\bot$ is clearly invariant
with respect to $\mathcal{L}^\Theta_x(\mathbb{M})$-asimulations
and we have $\bot \in \mathbb{M}$. Throughout this proof, we will
write $Con(\varphi(x))$ for the following set:
$$
\{ \psi(x) \in \mathcal{L}^\Theta_x(\mathbb{M}) \mid \varphi(x)
\models \psi(x) \}
$$

Our strategy will be to show that $Con(\varphi(x)) \models
\varphi(x)$. Once this is done, we will apply compactness of
first-order logic and conclude that $\varphi(x)$ is equivalent to
a finite conjunction of formulas in
$\mathcal{L}^\Theta_x(\mathbb{M})$ and hence, since we have
$\wedge \in \mathbb{M}$, to a formula in
$\mathcal{L}^\Theta_x(\mathbb{M})$.

To show this, take any $\Theta$-model $M_1$ and  $a \in U_1$ such
that $a \models_1 Con(\varphi(x))$. Then, of course, we also have
$Con(\varphi(x)) \subseteq Tr(\mathcal{L}^\Theta_{x}(\mathbb{M}),
M_1, a)$. Such a model exists, because $\varphi(x)$ is
satisfiable, and $Con(\varphi(x))$ will be satisfied in any model
satisfying $\varphi(x)$. Then we can also choose a $\Theta$-model
$M_2$ and $b \in U_2$ such that $b \models_2 \varphi(x)$ and

\begin{equation}\label{E:subset}
Tr(\mathcal{L}^\Theta_{x}(\mathbb{M}), M_2, b) \subseteq
Tr(\mathcal{L}^\Theta_{x}(\mathbb{M}), M_1, a),
\end{equation}

so that we get

\begin{equation}\label{E:inclusion}
\{ \langle a,b\rangle \} \in
Rel(\mathcal{L}^\Theta_{x}(\mathbb{M}), M_1, M_2).
\end{equation}

For suppose otherwise. Then for any $\Theta$-model $M$ such that
$U \subseteq \omega$ and any $c \in U$ such that $M, c \models
\varphi(x)$ we can choose a formula $\chi_{(M, c)}$ in
$\mathcal{L}^\Theta_{x}(\mathbb{M})$ such that $\chi_{(M, c)}$ is
in $Tr(\mathcal{L}^\Theta_{x}(\mathbb{M}), M, c)$ but not in
$Tr(\mathcal{L}^\Theta_{x}(\mathbb{M}), M_1, a)$. Then consider
the set
\[
S = \{\,\varphi(x)\,\} \cup \{\,\neg \chi_{(M, c)}\mid M, c
\models \varphi(x)\,\}
\]
Let $\{\,\varphi(x), \neg\chi_1\ldots , \neg\chi_q\,\}$ be a
finite subset of this set. If this set is unsatisfiable, then we
must have $\varphi(x) \models \chi_1\vee\ldots \vee \chi_q$, but
then we will also have
$$
\chi_1\vee\ldots \vee \chi_q \in Con(\varphi(x)) \subseteq
Tr(\mathcal{L}^\Theta_{x}(\mathbb{M}), M_1, a),
$$
and hence $\chi_1\vee\ldots \vee \chi_q$ will be true at $(M_1,
a)$. But then at least one of $\chi_1,\ldots, \chi_q$ must also be
true at $(M_1, a)$, which contradicts the choice of these
formulas. Therefore, every finite subset of $S$ is satisfiable,
and, by compactness, $S$ itself is satisfiable as well. But then,
by the L\"{o}wenheim-Skolem property, we can take a
$\Sigma_\varphi$-model $M'$ such that $U' \subseteq \omega$ and $g
\in U'$ such that $S$ is true at $(M',g)$ and this will be a model
for which we will have both $M', g \models \chi_{(M',g)}$ by
choice of $\chi_{(M',g)}$ and $M',g \not\models \chi_{(M',g)}$ by
satisfaction of $S$, a contradiction.

Therefore, we will assume in the following that some
$\Theta$-model $M_2$ and some $b \in U_2$ are such that $a
\models_1 Con(\varphi(x))$, $b \models_2 \varphi(x)$, and,
moreover \eqref{E:subset} and thus \eqref{E:inclusion} are
satisfied. According to Lemma \ref{L:ext}, there exist
$\omega$-saturated elementary extensions $M'$, $M''$ of $M_1$ and
$M_2$, respectively. We have:
\begin{align}
&M_1, a \models \varphi(x) \Leftrightarrow M', a \models
\varphi(x)\label{E:m1}\\
&M'', b \models \varphi(x)\label{E:m2}
\end{align}
Also, since $M_1$, $M_2$ are elementarily equivalent to $M'$,
$M''$, respectively, we have by \eqref{E:subset}:
\[
Tr(\mathcal{L}^\Theta_{x}(\mathbb{M}), M'', b) =
Tr(\mathcal{L}^\Theta_{x}(\mathbb{M}), M_2, b) \subseteq
Tr(\mathcal{L}^\Theta_{x}(\mathbb{M}), M, a) =
Tr(\mathcal{L}^\Theta_{x}(\mathbb{M}), M', a).
\]
Therefore, we have
$$
\{ \langle a,b\rangle \} \in
Rel(\mathcal{L}^\Theta_{x}(\mathbb{M}), M', M''),
$$
and further
$$
\langle a,b\rangle \in A = \bigcup
Rel(\mathcal{L}^\Theta_{x}(\mathbb{M}), M', M'').
$$
Therefore, $A$ is non-empty, and by $\omega$-saturation of $M'$,
$M''$ and Corollary \ref{L:sat}, $A$ is a
$\mathcal{L}^\Theta_{x}(\mathbb{M})$-asimulation. But then, by
\eqref{E:m2} and invariance of $\varphi(x)$ w.r.t.
$A\sigma(\mathcal{L}^\Theta_x(\mathbb{M}))$, we get that $M', a
\models \varphi(x)$, and further, by \eqref{E:m1} we conclude that
$M_1, a \models \varphi(x)$. Therefore, $\varphi(x)$ in fact
follows from $Con(\varphi(x))$.
\end{proof}

\section{Conclusion}\label{S:Con}

The generalization of the original result by J. van Benthem
achieved by Theorem \ref{L:final}, even if not exactly sweeping,
appears to be reasonably wide. Not only does it include all the
guarded connectives of degree $\leq 1$, but it also covers the
group of regular connectives of degree $2$, which seems to be rich
enough for many practical purposes. For example, given that among
the binary Boolean functions there are no functions which are both
$TFT$ and $FTF$, it follows that for every non-constant binary
Boolean function $f$ there exists at least one regular guarded
connective of degree $2$ with $f$ as a core, which consequently
can be handled by asimulations.

We would like to keep this paper within reasonable space limits
and so only presented here the main semantic characterization
result without indicating the standard ramifications that these
type of results tend to have. However, we do not see any obstacles
to getting these ramifications proved by an appropriate
modification of the above proofs. Thus, one can rather
straightforwardly prove for an arbitrary standard fragment
$\mathcal{L}^\Theta_{x}(\mathbb{M})$ that a formula $\varphi(x)$
is equivalent to a formula $\psi(x) \in
\mathcal{L}^\Theta_{x}(\mathbb{M})$ \emph{over a first-order
definable class $\varkappa$ of models} iff $\varphi(x)$ is
invariant w.r.t. to the class
$$
\bigcup\{ A\sigma(\mathcal{L}^\Theta_x(\mathbb{M}), M_1, M_2)\mid
M_1, M_2 \in \varkappa \},
$$
arguing along the lines of \cite[Theorem 7]{Ol13}, \cite[Theorem
6]{Ol14}, or \cite[Theorem 5]{Ol15}.

In much the same way, there seem to be no principal difficulties
in obtaining a `parametrized' version of Theorem \ref{L:final}
similar to \cite[Theorem 2]{Ol13}, \cite[Theorem 2]{Ol14}, or
\cite[Theorem 1]{Ol15}.

Another limitation of the above presentation is the finite
cardinality of $\mathbb{M}$ in the standard fragment
$\mathcal{L}^\Theta_{x}(\mathbb{M})$. It appears that a
generalization of the above proofs at least to reasonably small
infinite cardinalities is possible and straightforward.

One could also think of generalizing Theorem \ref{L:final} onto
the connectives guarded by relations of arity greater than $2$.
This can be interesting in connection with the possible
achievement of model-theoretic characterizations of e.g. sets of
standard translations induced by relevance logics. This problem
does not seem to be very difficult, although we cannot provide any
such generalization offhand.

Also, our main result can be easily extended onto non-standard
guarded fragments of special form. Generally speaking, for every
given guarded connective $\mu \in \mathbb{M}$ one must also demand
that $\mu^1,\ldots, \mu^{\delta(\mu) - 1} \in \mathbb{M}$. For
regular guarded connectives this condition can be weakened in
various ways. For instance, when dealing with modalities it is
sufficient that the series of $\mu$'s ancestors present in
$\mathbb{M}$ starts with $\mu^2$ and every gap in it contains at
most $1$ ancestor.

However, the most natural and tricky question is whether Theorem
\ref{L:final} in its most general form can be extended onto
guarded fragments, containing at least some types of non-standard
connectives, that is to say, without adding any other conditions
on the form of these fragments. It is not exactly clear  how many
non-standard guarded connectives of degree $2$ can be taken on
board in this way, although we conjecture that their number should
not be very big. Extension of Theorem \ref{L:final} onto the
degrees exceeding $2$, if possible at all, appears to require
methods different from the ones employed in the present paper.
Indeed, when considering guarded connectives of higher degrees,
one can no longer rely on distributivity properties which turned
out to be pivotal for our arguments about the guarded connectives
of degree $2$.

\section{Acknowledgements}

To be inserted

\bibliographystyle{habbrv}

\bibliography{intbib}

\begin{thebibliography}{1}

\bibitem{AlShk06}
N.~Alechina and D.~Shkatov.
\newblock A general method for proving decidability of intuitionistic modal
  logics.
\newblock {\em Journal of Applied Logic}, 04:291--230, 2006.

\bibitem{Badia2015}
G.~Badia.
\newblock Bi-simulating in bi-intuitionistic logic.
\newblock Available at: http://otago.academia.edu/GuillermoBadia/Papers, 2015.

\bibitem{Blackburn2006}
P.~Blackburn, J.~F. A. K.~v. Benthem, and F.~Wolter.
\newblock {\em Handbook of Modal Logic, Volume 3 (Studies in Logic and
  Practical Reasoning)}.
\newblock Elsevier Science Inc., New York, NY, USA, 2006.

\bibitem{ChK73}
C.~Chang and H.~Keisler.
\newblock {\em Model Theory}.
\newblock North Holland, 1973.

\bibitem{Doets96}
K.~Doets.
\newblock {\em Basic Model Theory}.
\newblock CSLI Publications, 1996.

\bibitem{Ol13}
G.~Olkhovikov.
\newblock Model-theoretic characterization of intuitionistic propositional
  formulas.
\newblock {\em Review of Symbolic Logic}, 06:348--365, 2013.

\bibitem{Ol14}
G.~Olkhovikov.
\newblock Model-theoretic characterization of intuitionistic predicate
  formulas.
\newblock {\em Journal of Logic and Computation}, 24:809--829, 2014.

\bibitem{Ol15}
G.~{Olkhovikov}.
\newblock {Expressive power of basic modal intuitionistic logic as a fragment
  of classical FOL}.
\newblock {\em ArXiv e-prints}, Nov. 2015, 1511.01730.

\end{thebibliography}

 }
\end{document}